\documentclass{article}
\usepackage{amssymb}
\usepackage{verbatim}
\usepackage{amsmath}
\usepackage{amsfonts}
\usepackage{amsthm}
\usepackage{xcolor}
\usepackage{array}
\usepackage{enumerate}
\usepackage{graphicx}
\usepackage[version=4]{mhchem}
\usepackage{url} 
\usepackage{tikz-cd}

\newtheorem{remark}{Remark}
\newtheorem{definition}{Definition}
\newtheorem{example}{Example}
\newtheorem{corollary}{Corollary}
\newtheorem{lemma}{Lemma}
\newtheorem{proposition}{Proposition}

\begin{document}


\author{Justin Eilertsen\\
        Mathematical Reviews\\ American Mathematical Society\\
        416 4th Street\\Ann Arbor, MI, 48103\\
        e-mail: {\tt jse@ams.org}\\\\
        Valery G.\ Romanovski\\
        Faculty of Electrical Engineering and Computer Science \\ University of Maribor, 
        Koro\v ska cesta 46, SI-2000 Maribor, Slovenia\\
        Center for Applied Mathematics and Theoretical Physics \\
        Mladinska 3,  SI-2000 Maribor, Slovenia\\
        Faculty of Natural Science and Mathematics  \\
        University of Maribor, Koro\v ska cesta 160, SI-2000 Maribor, Slovenia\\
        email:{\tt valerij.romanovskij@um.si}    \\\\
        Santiago Schnell\\
        Department of Mathematics \\ 
        Dartmouth, Hanover, NH 03755\\
        Department of Biochemistry \& Cell Biology, and \\
        Department of Biomedical Data Sciences\\
        Geisel School of Medicine at Dartmouth, Hanover, NH 03755\\
        e-mail: {\tt santiago.schnell@dartmouth.edu}\\\\
        Sebastian Walcher\\
        Fachgruppe Mathematik, RWTH Aachen\\
        D-52056 Aachen, Germany\\
        e-mail: {\tt walcher@mathga.rwth-aachen.de}
        }

\title{Lumping of reaction networks: \\Generic and critical parameters}

\maketitle

\begin{abstract}
We investigate linear lumping for parameter-dependent mass action reaction networks, 
distinguishing between generic and critical parameter regimes. For generic parameters---those 
ranging in some non-empty open subset of parameter space---we prove that exact linear lumping 
yields only ``obvious'' reductions: elimination of non-reactant species or projections along
stoichiometric first integrals. This characterization extends to reaction networks with 
product-form kinetics, including Michaelis--Menten and Hill-type rate laws. For mass action systems we proceed to develop 
an algorithmic approach to identify critical parameter sets---algebraic subvarieties in parameter 
space where non-trivial lumpings become available. This procedure reduces the determination of lumping maps to a system of finitely 
many polynomial equations. It also applies to constrained lumping scenarios (which are frequently motivated 
by chemical considerations). We then review and extend results about proper lumpings. Finally, we discuss lumpings of a self-replicator system, and of a two-pathway enzyme mechanism, to document the viability of our methods in relevant scenarios. Our results clarify the 
relationship between structural (parameter-independent) and fine-tuned (parameter-dependent)
reductions, with implications for approximate lumping when system parameters lie near 
critical values.
\end{abstract}

\section{Introduction}

Reaction networks from chemistry and biochemistry give rise to parameter-dependent polynomial 
ordinary differential equations that may involve many variables, ``$x\in \mathbb{R}^n$" (concentrations of chemical 
species) and many parameters, ``$k\in \mathbb{R}^d$" (such as rate constants) and is of the general form
\begin{equation}\label{basicode}
\dot{x} = F(x,k),  \quad F: \mathbb{R}^n\times \mathbb{R}^d\to \mathbb{R}^n
\end{equation}
where ``$\dot{\phantom{x}}$" denotes differentiation with respect to time.

From both theoretical and practical perspectives, 
it is desirable to construct related systems of smaller dimension that are more amenable to 
mathematical analysis, numerical simulation, and parameter identification. Dimension reduction 
methods provide a toolbox for such simplifications, with two fundamentally different 
approaches: On the one hand, one aims to identify (locally attracting) distinguished invariant {\it submanifolds} of $\mathbb{R}^n$. On the other hand, one aims to find {\it lumping maps}, which aggregate variables in such a way that the system of aggregated variables satisfies a differential equation in its own right. The central topic of the present manuscript is to investigate {\it linear} lumping maps.  Our main results will be concerned with the characterization and construction of such maps in different parameter regimes.

\subsection{Background and some history}
In biochemistry, approximations via invariant manifolds are frequently referred to as quasi-steady-state approximations (QSSA); see e.g. Segel and Slemrod~\cite{SegelSlemrod1989,Shoffner}. They include time scale separation methods from singular perturbation theory (Fenichel \cite{Fenichel1979}),
computational singular perturbation (CSP); see Lam and Goussis \cite{LamGoussis1994}, and intrinsic low-dimensional 
manifolds (ILDM); see Maas and Pope \cite{MaasPope1992}. For a variant involving several time scales see e.g.\ Cappelletti and Wiuf \cite{CapWi2017}.
Goeke, Walcher, and Zerz~\cite{GWZ2015} discuss the relation and distinction between QSSA and singular 
perturbation reduction. Time-scale methods produce \emph{approximate} reductions that are valid asymptotically 
as the time-scale ratio becomes extreme~\cite{Fenichel1979, LamGoussis1994, MaasPope1992}.

The contrasting lumping approach aggregates variables 
via algebraic conditions that hold globally. The notion of lumping may admit several interpretations depending on authors' backgrounds and 
objectives. In most settings, lumping refers to an aggregation into a smaller 
set of new variables that satisfy a closed system of differential equations. More specifically, one looks for a pair $(\Phi,\,G)$ with a smooth map, $\Phi$, from $\mathbb{R}^n$ to $\mathbb{R}^m$ with $m < n$, with the rank of its Jacobian generically equal to $m$, which sends the solutions of \eqref{basicode} to solutions of a reduced system
\begin{equation}\label{basicredode}
\dot{y} =  G(y,k) \text{  on  } \mathbb R^m.
\end{equation}
In this situation we will speak of $\Phi$ as a solution-preserving map from \eqref{basicode} to \eqref{basicredode}.

If (for instance) the rank of the Jacobian is maximal everywhere, then every level set $\Phi(x)={\rm const.}$ is a submanifold of $\mathbb R^n$. Then to any solution of the reduced system there corresponds a``submanifold moving with time" of the original system, and in particular for every invariant set of the reduced equation, its inverse image is an invariant set of the original equation. \\
In the present paper, we focus on linear lumping maps $T:\, \mathbb R^n\to\mathbb R^m$ of maximal rank; the submanifolds are then affine subspaces. In this setting the existence condition for 
a reduced system reads
$TF(x,k)=G(Tx,k)$.

 We distinguish 
three main variants:

\medskip
\noindent\textbf{Exact lumping.} 
A lumping is \emph{exact} if every solution of the original system $\dot x=F(x)$
maps to a solution of a certain reduced system, with no approximation error. The classical work of 
Wei and Kuo~\cite{WK1969a} on linear lumping for monomolecular reactions falls in this 
category, establishing the fundamental correspondence between linear lumping maps and 
invariant subspaces of the kinetic matrix. Li and Rabitz~\cite{LR1989} extended these ideas 
to nonlinear systems, showing that exact linear lumping requires the kernel of the lumping 
map to be invariant under all Jacobians $DF(x)$. Li, Rabitz, and T\'oth~\cite{LRT1994} further 
discussed nonlinear lumping transformations. 

\medskip
\noindent\textbf{Constrained lumping.} 
In applications, one often prescribes that certain 
observables (linear combinations of concentrations) must appear in the reduced system. 
Constrained lumping seeks the maximal reduction compatible with these constraints. 
Li and Rabitz~\cite{LR1991a,LR1991b} developed systematic approaches, and (for instance) the recent CLUE 
algorithm by Ovchinnikov et al.~\cite{OVPT2021} provides an efficient computational 
implementation.

\medskip
\noindent\textbf{Approximate lumping.} 
When exact lumping conditions fail---the typical situation 
in practice---one may relax the algebraic conditions to obtain approximate reductions. 
Wei and Kuo~\cite{WK1969b} initiated this direction for monomolecular systems, while 
Li and Rabitz~\cite{LR1990} developed the general theory using Luenberger observer methods, 
extended to nonlinear maps in Li et al.\ \cite{LTRT1994}. From dependency theorems for ordinary differential equations (for instance in Walter \cite{Walter1998}; see also 
Leguizamon-Robayo et al.~\cite{LJTTV}) one obtains rigorous error bounds: if 
the lumping condition is relaxed by tolerance $\varepsilon$, the approximation error on any compact time interval is 
$O(\varepsilon)$. 

The algebraic theory of lumping for linear systems is based on Wei and Prater~\cite{WP1962}, who studied 
reversible first-order networks satisfying detailed balance. For such systems, the kinetic 
matrix possesses real, non-positive eigenvalues and symmetry in an appropriate inner 
product---properties enabling diagonalizability and systematic decomposition. Wei and Kuo~\cite{WK1969a} then 
established a definitive characterization for monomolecular systems (not necessarily satisfying detailed balance): a linear map $T$ 
defines an exact lumping for $\dot{x} = Kx$ if and only if the row space of $T$ is 
$K^{\rm tr}$-invariant. When eigenvalues are distinct, any subspace spanned by eigenvector subsets 
is invariant; repeated eigenvalues may require Jordan block analysis. This eigenspace 
characterization connects lumping to spectral theory and Markov chain 
aggregation (see Kemeny and Snell \cite{KeSn1960}).

The extension to nonlinear systems by Li and Rabitz~\cite{LR1989} revealed that lumpability 
depends on \emph{joint} invariant subspace structure: $T$ is a lumping map for 
$\dot{x} = F(x)$ if and only if the row space of $T$ is invariant under $DF(x)^{\rm tr}$ for 
\emph{all} $x$. For polynomial systems, this poses a finite linear algebra problem. 
T\'oth et al.~\cite{TLRT1997} noted some dynamical implications of the fact that lumping preserves 
invariant sets.

Recent computational advances have made linear lumping practically tractable for large systems. 
The CLUE algorithm (Ovchinnikov et al.\ \cite{OVPT2021}) computes maximal exact reductions for polynomial ODEs,
handling systems with thousands of variables. The ERODE framework (Cardelli et al. \cite{CardelliTTV2017}) 
approaches reduction through partition refinement, connecting to bisimulation theory from 
computer science. These tools have enabled systematic assessment of lumpability across 
model databases, revealing that over 64\% of models in the BioModels database\footnote{See {\tt https://www.biomodels.org}.} admit exact lumping 
reductions; see P\'erez Verona at al.\ \cite{PerezVerona2021}.

Reaction networks are 
inherently parameter-dependent, with rate constants that may be known precisely, estimated 
from data, or treated as free parameters. Prior work has largely focused on two extremes.
On one hand, \emph{parameter-independent lumping} seeks reductions valid for \emph{all} 
parameter values, as implemented in CLUE~\cite{OVPT2021}; such reductions are 
``structural'', determined by network topology alone.  On the other hand, 
\emph{fixed-parameter lumping} finds reductions for specific numerical parameter values, 
as in classical Wei-Kuo theory. Between these extremes lies unexplored territory: How does 
lumpability vary as parameters change? For which parameter values do non-trivial lumpings 
exist? This is the {\it parametric perspective}, and from this viewpoint we pursue three main objectives. First, we 
characterize generic lumping (which we write shortly for a lumping valid in a given network
for generic parameters; i.e.,\ those ranging in some non-empty
open subset of parameter space). We will show that exact linear lumping yields only ``obvious'' reductions—--elimination of 
non-reactant species or projections along stoichiometric first integrals. This explains 
why structural lumping often fails to find remarkable reductions. Second, we develop an 
algorithmic approach to identify critical parameters where non-trivial lumpings become 
available. For a first impression of these notions, consider the system
\[
\begin{array}{rcl}
\dot x_1&=& -k x_1  \\
\dot x_2&=&kx_1-\ell x_2
\end{array}
\]
with positive parameters $k,\,\ell$, which models a network where species $X_1$ converts irreversibly to species $X_2$, while $X_2$ degrades irreversibly. Here, the map that sends $x=\begin{pmatrix} x_1\\ x_2\end{pmatrix}$ to $x_1$ is a lumping map for any choice of parameters and thus represents a generic scenario. On the other hand, the map\footnote{Both lumping maps correspond to left eigenvectors of the underlying matrix.} sending $x$ to $ax_1+x_2$, for some fixed given $a$, is a lumping map if and only if $k\not=\ell$ and $k/(k-\ell)=a$; this is a critical parameter condition.
 Critical parameters  lie on semi-algebraic subvarieties in parameter space, and their determination 
in principle (up to size-related feasibility matters) reduces to solving finitely many polynomial systems. 

As will be seen, lumping methods provide a tool for discovering special structure in reaction networks. These may include ``hidden conservation laws'' (i.e., additional first integrals) at critical parameter values, which still hold approximately at nearby parameter values, and generally particular invariant sets.

\subsection{Overview of results}

The main contributions are summarized as follows. \textbf{Section 2} establishes the mathematical framework. We consider parameter-dependent polynomial ODEs $\dot{x} = F(x,k)$ with $x \in \mathbb{R}^n$ and $k \in \mathbb{R}^d$, and recall the Li-Rabitz criterion (Proposition \ref{HWprop}) that $T$ is a linear lumping map if and only if $\ker T$ is invariant under all Jacobians $DF(x,k^*)$. This reduces lumpability to a question about joint invariant subspaces.

\textbf{Section 3} analyzes lumping for generic mass action networks and some generalizations. For a single mass action reaction we show in Proposition \ref{prop:onereac} that invariant subspaces are of two types: Type 1 corresponding to non-reactant species, and Type 2 corresponding to stoichiometric first integrals. For reaction networks with generic parameters (Proposition \ref{prop:all}), a lumping map must be a lumping for each individual reaction. The resulting characterization (Proposition \ref{prop:mixlump}, Corollary \ref{cor:mixlump}) shows that generic lumping yields only reductions by eliminating common non-reactant species or using common stoichiometric first integrals. We provide a construction algorithm and extend these results to product-form kinetics including Michaelis--Menten and Hill-type rate laws. On this basis, one obtains a simplification of the CLUE algorithm.

\textbf{Section 4} develops the theory of critical parameters. Given a candidate lumping map $T$, we determine necessary and sufficient conditions on parameters $k^*$ for $T$ to be solution-preserving to a system of smaller dimension (Lemma~\ref{lalem}). These conditions form a system of polynomial equations in the rate parameters and entries of $T$. Thus finding all critical parameters reduces to solving finitely many such systems, one for each choice of independent columns in the row-echelon form of $T$. The approach extends to constrained lumping (Remark~\ref{remconstr}), where some rows of $T$ are prescribed. As can be expected for systems of polynomial equations, their size may lead to feasibility problems. Three worked examples illustrate the method: a three-species first-order network, the reversible Michaelis--Menten system, and a constrained reduction of Michaelis--Menten preserving stoichiometric first integrals. 

\textbf{Section 5} studies proper lumping and symmetry-based approaches. For \emph{proper lumping}---where species partition into blocks and each macro-variable sums concentrations within a block (see Wei and Kuo \cite{WK1969a}, and also Cardelli et al. \cite{CardelliTTV2017})---we establish a  column-sum criterion (Proposition \ref{colslem}): $k^*$ is critical if and only if all column sums within each Jacobian block are equal. We then investigate lumpings that are motivated by the chemical assumption that certain species behave (dynamically) alike. A mathematical interpretation of this assumption leads to species permutations that respect complexes, and thus to graph automorphisms. Taking a further step, restricting the orbit space reduction to linear invariants will identify critical parameters. 

\textbf{Section 6} applies the theory to a self-replication model from origin-of-life chemistry, and to a two-pathway enzyme system. The purpose is to document and illustrate the applicability of our theoretical framework to relevant systems. We add a few examples to indicate that the reduction reveals mathematically and biologically interesting features. In particular, at some critical values, ``hidden conservation laws'' may emerge, confining the dynamics to a lower-dimensional manifold, and small perturbations of critical parameters may lead to interesting dynamical behavior. We only sketch these applications in the present work; a thorough discussion will be the subject of a future paper. 

\textbf{Section 7} closes the paper with a discussion and a view toward future work.


\section{Setting}

Our focus lies on parameter-dependent ordinary differential equations
\begin{equation}\label{parode}
\dot x=F(x,k), \quad x\in \mathbb R^n,\, k\in \mathbb R_{\geq0}^d,
\end{equation}
where $F:\,\mathbb R^n\times\mathbb R^d\to \mathbb R^n$ is a polynomial map. In some instances, we will also include scenarios with $F$ analytic on an open subset of $\mathbb R^n\times\mathbb R^d$. Moreover, for fixed $k^*\in \mathbb R^d$ we call
\begin{equation}\label{parodes}
\dot x=F(x,k^*), \quad x\in \mathbb R^n,
\end{equation}
the {\em specialization} of \eqref{parode} at the parameter value $k^*$.

We are interested in the existence of linear lumping maps for such parameter-dependent systems.  Thus, consider  a linear map
\begin{equation}\label{Tdef}
\mathbb R^n\to \mathbb R^e, \quad e<n, \quad x\mapsto Tx, \quad{\rm rank}\,T=e.
\end{equation}
Then, by a familiar criterion, $T$ defines a solution preserving map\footnote{Thus, for every solution $z(t)$  of  \eqref{parodes}, $Tz(t)$ is a solution of \eqref{redparodes}.}  from a specialization \eqref{parodes} to a polynomial system
\begin{equation}\label{redparodes}
\dot y = G(y, k^*)
\end{equation}
if and only if the following identity holds:
\begin{equation}\label{linlump}
TF(x,k^*)=G(Tx,k^*).
\end{equation}
If condition \eqref{linlump} is satisfied, then we call $T$ a {\em linear lumping map} for the parameter value $k^*$, and \eqref{redparodes} a {\em reduced system} for \eqref{parodes}.
\begin{remark}\label{nonuniqrem} {\em Non-uniqueness:
To every linear lumping map $T$ one has equivalent linear lumping maps $QT$ for every invertible $Q\in\mathbb R^{e\times e}$, with reduced system $\dot y=QG(Q^{-1}y,k^*)$. This fact reflects the possibility of basis changes in $\mathbb R^e$, or (in other words) the possibility to apply Gauss row operations for simplification}.  
\end{remark}
\begin{remark}\label{restrem}{\em 
Ovchinnikov and co-authors \cite{OVPT2021} consider \eqref{parode} as a differential equation in $\mathbb R^{n+d}$ for variables $(x,\,k)$, augmented by the additional equations $\dot k=0$. Thus, they are interested in simultaneous lumpings that are applicable for all parameters. But in their algorithms they only consider lumping maps that act on $x$ alone.
}
\end{remark}
Linear lumping maps are subject to rather strong restrictions; Li and Rabitz \cite{LR1989} established a criterion involving the Jacobians\footnote{The Jacobian will always be taken with respect to the variable $x$.} $DF(x, k^*)$, $x\in\mathbb R^n$. 
For monomolecular reaction networks, thus $F$ linear, see the earlier seminal work \cite{WK1969a} by Wei and Kuo\footnote{Wei and Kuo discussed further restrictions on lumpings, to ensure that the reduced system again admits an interpretation via a reaction network.}.

A relatively convenient criterion was given in Hadeler and Walcher~\cite{HW2006}. We include a proof here, for the sake of completeness.
\begin{proposition}\label{HWprop}
Given \eqref{parode}, a surjective linear map $T$, as in \eqref{Tdef},  defines a linear lumping map for parameter $k^*$ if and only if 
\begin{equation}\label{kerx}
DF(x,k^*)\left(\ker T\right)\subseteq \ker T \text{  for all  }x\in\mathbb R^n.
\end{equation}
\end{proposition}
\begin{proof}
\begin{enumerate}[(i)]
\item We first show: There exists $G$ such that \eqref{linlump} is satisfied if and only if
\begin{equation}\label{xplusz}
TF(x+z,k^*)=TF(x,k^*)\quad \text{for all  } z\in\ker T,\,\text{  all  }x\in \mathbb R^n.
\end{equation}
Necessity of this condition is obvious from
\[
TF(x+z, k^*)=G(T(x+z),k^*)=G(Tx,k^*)=TF(x,k^*).
\]
Conversely, given $y\in \mathbb R^e$, there exists $x\in\mathbb R^n$ such that $y=Tx$ by surjectivity,  and with \eqref{xplusz} one sees that
\[
G(y,k^*):=TF(x,k^*)
\]
is well defined, and \eqref{linlump} holds.
\item 
To show equivalence of \eqref{xplusz} and \eqref{kerx}, we consider the Taylor expansion of \eqref{xplusz}, thus
\[
F(x+z,k^*)=F(x,k^*)+DF(x,k^*)z+\cdots+ \frac1{m!}D^mF(x,k^*)(z,\ldots,z),
\]
with $F$ of degree $m$.
Assuming $TF(x+z,k^*)=TF(x,k^*)$ for all $x\in \mathbb R^n$ and $z\in \ker T$, replace $z$ by $\lambda z$, $\lambda\in\mathbb R$ and compare degrees in $\lambda$ to see that all $TD^jF(x,k^*)(z,\ldots,z)=0$. For the reverse direction, let $z\in\ker T$ and differentiate $TDF(x,k^*)z=0$ with respect to $x$ to obtain

\[
TD^jF(x,k^*)(z,w_1,\ldots,w_{j-1})=0 \text{  for all  } w_1,\ldots, w_{j-1}\in\mathbb R^n; \quad 2\leq j\leq m,
\]
which implies $TD^jF(x,k^*)(z,\ldots,z)=0$ for all $j$.
\end{enumerate}
\end{proof}
\begin{remark} {\em  The statement and its proof remain locally valid for differential equations with analytic right hand side and thus globally valid on every connected component of the domain of definition. In particular this applies to rational right hand side when defined on the whole positive orthant.
The necessity of condition \eqref{kerx} was observed by Li and Rabitz \cite{LR1989}; later Li et al. \cite{LRT1994} also showed sufficiency.}
\end{remark}
The following -- equivalent -- criterion also goes back to Li and Rabitz \cite{LR1989}. It forms the basis for the computations in Ovchinnikov et al. \cite{OVPT2021}.
\begin{corollary}\label{HWcor}
Given \eqref{parode}, a surjective linear map $T$ defines a linear lumping map for the parameter $k^*$ if and only if\footnote{We denote the transpose by $\cdot ^{\rm tr}$.}
\begin{equation}\label{kerxtrans}
DF(x,k^*)^{\rm{tr}}\left({\rm im}\, T^{\rm tr}\right)\subseteq {\rm im}\, T^{\rm tr}\text{  for all  }x\in\mathbb R^n.
\end{equation}
\end{corollary}
\begin{proof}
From \eqref{kerx} we obtain the necessary and sufficient criterion
\begin{equation*}
DF(x,k^*)^{\rm tr}\left((\ker T)^\perp\right)\subseteq (\ker T)^\perp \text{  for all  }x\in\mathbb R^n
\end{equation*}
by passing to the dual space. The assertion follows with ${\rm im}\,T^{\rm tr}=(\ker T)^\perp$.
\end{proof}
\begin{remark}\label{rem:construct} {\em 
Proposition \ref{HWprop} and Corollary \ref{HWcor} provide access to a construction of lumping maps as follows: Given a subspace $W\subseteq \mathbb R^n$ that is invariant for all $DF(x,\,k^*)$, choose a basis $v_1^{\rm tr},\ldots, v_m^{\rm tr}$ of $W^\perp$  (viewed as a subspace of the row space $\mathbb R^{1\times n}$), and take $T$ as the matrix with rows $v_1^{\rm tr},\ldots, v_m^{\rm tr}$. This works because the column space of $T^{\rm tr}$ (i.e., the row space of $T$) equals $W^\perp$. The freedom of choice for the basis is reflected in Remark \ref{nonuniqrem}.}
\end{remark}

\begin{remark}\label{redcomprem}{\em 
The following observations open a path toward computing the reduced system: Let $\lambda_1(x),\ldots,\lambda_e(x)$ be the entries of $Tx$. Then there exists $G$ such that \eqref{linlump} holds if and only if the entries of $TF(x,k^*)$ are $\mathbb R$-linear combinations of monomials in $\lambda_1, \ldots,\lambda_e$. We state this in a more formal manner.
\begin{enumerate}
    \item 
Setting
$y_i=\lambda_i(x)$, $1\leq i\leq e$,
at a critical parameter $k^*$ there exist polynomials $\gamma_i$ such that
\[
\dot y_i=\lambda_i(F(x,k^*))=
\gamma_i(y_1,\ldots,y_e,k^*),\qquad 1\leq i\leq e.
\]

As noted in Li and Rabitz \cite{LR1989}, Section 2, equation (9), determining the $\gamma_i$ may be seen as a linear algebra problem in the finite dimensional space of polynomials of bounded degree.
\item One can take a different perspective: For a given critical parameter $k^*$ we have polynomials 
\[
\mu_j(x,k^*)=\lambda_j(F(x,k^*)),\qquad 1\leq j\leq e,
\]
and
\[
y_j:=\lambda_j(x),\qquad v_j:=\mu_j(x,k^*),\qquad 1\leq j\leq e.
\]
Elimination theory states that $x_1,\ldots,x_n$ can be eliminated from the above system, and moreover there exist efficient algorithms for this purpose: Briefly, with $y=(y_1,\ldots,y_e)$ and $v=(v_1,\ldots,v_e)$ as independent indeterminates, formulate the graph ideal
\[
{\mathcal J}=\left< y_j-\lambda_j(x), v_j-\mu_j(x, k^*); 1\leq j\leq e\right>\subseteq \mathbb R[x,y,v].
\]
Eliminating $x$ yields relations between $y$ and $v$. At a critical parameter, these are the graph relations $v_j=\gamma_j(y,k^*)$, which give the reduced vector field.
(See for instance Cox et al.\ \cite{CLOS} for more details.) Thus we have the $\mu_j$ as polynomials $\gamma_j(y_1,\ldots,y_e,\,k^*)$, which form the right-hand side of the reduced system.
\item Alternatively, one may obtain a ``lumping-adapted'' version of the full system, by which we mean the following: First, complete $y_1,\ldots, y_e$ (e.g. by suitable $x_j$) with $y_{e+1},\ldots, y_n$ to a basis of $\mathbb R^{1\times n}$. Then rewrite system \eqref{parodes} in the new coordinates $y_1,\ldots,y_n$. (This requires to invert a matrix of size $n\times n$.) Since the coordinate change is applicable for all parameters, this procedure also yields a representation of the system when $k^*$ is perturbed by a small parameter.
\end{enumerate}
}
\end{remark}

\section{Lumping for generic reaction networks}\label{sec:genlump}
In the present section we will concentrate on mass action kinetics, but we add some observations on general kinetics in the penultimate  subsection. 
We distinguish lumping maps for a single specialization $\dot x=F(x,k^*)$ from simultaneous lumping maps that reduce $\dot x=F(x,k)$ for all $k$ in a nonempty open subset of parameter space. In the latter case we will speak of lumping maps for the reaction network, or -- to emphasize -- of the generic reaction network.
\subsection{Review of mass action networks }
We recall some basics about reaction networks; for more the reader is referred to the monograph \cite{Fein} by Feinberg.\\
A {\em mass-action chemical reaction network} $({\mathcal X},\,{\mathcal Y},{\mathcal R}, k)$ consists of the following ingredients:
\begin{itemize}
\item A finite set of {\em species} $\mathcal{X}=\{X_1,\dots,X_n\}$, with concentrations $x_1,\ldots,x_n$, respectively.
\item A finite set $\mathcal Y$ of {\em complexes}. By definition, every complex has the form
\[
 Y=\sum_{i=1}^n \alpha_{i}X_i ,\qquad \alpha_i\in\mathbb N_0, \ i=1,\ldots,n.
\]
\item A set ${\mathcal R}\subseteq {\mathcal Y}\times {\mathcal Y}$ of {\em reactions}. Formally,  a reaction is an ordered pair $(Y_j,Y_\ell)$ of complexes, but as usual we will write 
$
\ce{$Y_j$ -> $Y_\ell$}
$
to symbolize it.
\item To every reaction a nonnegative number $k_{\ell j}$, the {\em rate parameter}, is assigned; symbolically
\[
\ce{$Y_j$  ->[$k_{\ell j}$] $Y_\ell$}.
\]
 Thus a mass-action reaction network may be viewed as a directed graph with the complexes as vertices and the reactions as edges, which are labeled by the rate constants.
\item Because we assume mass action kinetics throughout, the time evolution of the concentrations in a single reaction 
\begin{equation*}
Y_1=m_1X_1+\cdots+m_nX_n, \quad Y_2=r_1X_1+\cdots+r_nX_n; \quad Y_1 \ce {->[$k$]} Y_2
\end{equation*}
is governed by the differential equation system
\begin{equation}\label{singlereac}
\frac{d}{dt}\begin{pmatrix} x_1\\ \vdots\\ x_n\end{pmatrix}=k\,\varphi(x)\,v;
\end{equation}
with 
\begin{equation}\label{singlereacplus}
\varphi(x)=x_1^{m_1}\cdots x_n^{m_n},\quad v=\begin{pmatrix} r_1-m_1\\ \vdots\\ r_n-m_n\end{pmatrix}=:\begin{pmatrix}\nu_1\\ \vdots\\ \nu_n\end{pmatrix}.
\end{equation}
To determine the time evolution of a reaction network, add up all the individual reaction terms on the right hand side. This {\em reaction equation system} therefore has the form
\begin{equation}\label{eq:ODE2Z}
    \dot x=\sum k_i\varphi_i(x)\,v_i.
\end{equation}
\item One calls a linear form $\mu$ a {\em stoichiometric first integral} of the reaction equation if it sends every $v_i$ to $0$. Since the $v_i$ have integer entries, it suffices to consider stoichiometric first integrals with integer coefficients.
\item Since all rate parameters are nonnegative, the positive orthant $\mathbb R_{\geq 0}^n$ is a positively invariant set for the reaction equations.
\end{itemize}
For a different representation of \eqref{eq:ODE2Z}, rename complexes as 
\[
Y_j=\sum_i y_{ij}X_i, 
\]
and write reactions as
\[
\ce{$Y_j$  ->[$k_{\ell j}$] $Y_\ell$}
\]
with rate constants $k_{\ell j}$.
Thus we obtain an equivalent version of the reaction equation system with the following ingredients:
\begin{itemize}
\item The \emph{complex matrix}, defined as
\begin{equation*}\label{massacmat}
Y=\left(y_{ij}\right)_{1\le i\le n,\, 1\le j\le d} \ \in \mathbb R^{n\times d}, 
\end{equation*}
thus it consists  of the stoichiometric coefficients of the complexes. Let $y_1,\dots,y_d$ denote its columns. 
\item The \emph{Laplacian matrix} { $A(k) = (a_{ij})_{1\le i,j\le d}\in \mathbb R^{d\times d}$}  has entries
\[ a_{ij} =k_{ij} \text{  whenever  } i\neq j,  \text{  and  }a_{jj} = -\sum_{\ell:\,j\neq \ell} {k_{\ell j}},\qquad \textrm{for }\quad i,j=1,\ldots,d, \]
where $k_{ij}=0$ if there is no reaction $Y_j\rightarrow Y_i$. 
\item Moreover abbreviate

\begin{equation*}
x^Y:=\left( \prod_{1\leq i\leq n} x_i^{y_{ij}}\right)_{1\leq j\leq d}.
\end{equation*}

\end{itemize}
Then system \eqref{eq:ODE2Z} can be restated in the form
\begin{equation}\label{eq:ODE2A}
\dot{x}=  Y A(k)\,  x^Y, \qquad x\in \mathbb R^n_{\geq 0}.
\end{equation}

\subsection{Lumpings of a single reaction}
We consider a single reaction, governed by equation \eqref{singlereac}. With no loss of generality one may set $k=1$ when $k>0$\footnote{In case $k=0$ every linear map of a given rank is a lumping map. }, and we will abbreviate the right hand side by $F(x)$.
According to Proposition \ref{HWprop} we need to determine all subspaces of $\mathbb R^n$ that are stable under every Jacobian $DF(x), \, x\in\mathbb R^n$. Now 
\[
DF(x)=v\,D\varphi(x)=\varphi(x)\begin{pmatrix}\nu_1\\ \vdots\\ \nu_n\end{pmatrix}\,\begin{pmatrix}\dfrac{m_1}{x_1},&\cdots,&\dfrac{m_n}{x_n}\end{pmatrix}.
\]
The following result, as well as its underlying approach, are new.
\begin{proposition}\label{prop:onereac} Let $W\subseteq\mathbb R^n$ be a subspace that is stable under the Jacobian $DF(x)$ for every $x$. Then one of the following holds:
\begin{itemize}
\item $W$ is of {\em Type 1}:  $W\subseteq \sum_{i:\,m_i=0} \mathbb Re_i$. (Note that the sum extends over all indices for which species $X_i$ is not a reactant.)
\item $W$ is of {\em Type 2}:  $v\in W$.
\end{itemize}
Conversely, all subspaces of Types 1 or 2 are stable under every Jacobian.
\end{proposition}
\begin{proof}
If there is some $w\in W$ such that $D\varphi(x)w\not=0$, then $v\in W$, so we have Type 2. On the other hand, every subspace that contains $v$ is obviously stable under all $DF(x)$. Otherwise, $D\varphi(x)w=0$ for all $x$, equivalently 
\[
\sum m_iw_i/x_i=0 \text{  for all  } x_1\not=0,\ldots,x_n\not=0.
\]
This sum of rational functions is identically zero if and only if every $m_iw_i/x_i=0$. Equivalently, $w_i=0$ whenever $m_i\not=0$ or, in other words, $w\in \sum_{i:\,m_i=0} \mathbb Re_i$.
\end{proof}
To obtain the reducing maps, determine $W^\perp\subseteq \mathbb R^{1\times n}$ according to  Remark \ref{rem:construct}. Note that Type 1 lumpings need not exist for single reactions, but Type 2 lumpings exist whenever $n\geq 2$. Figure~\ref{fig:type1_type2} illustrates the 
geometric meaning of these two types.
\begin{figure}[htbp]
\centering
\includegraphics[width=0.95\textwidth]{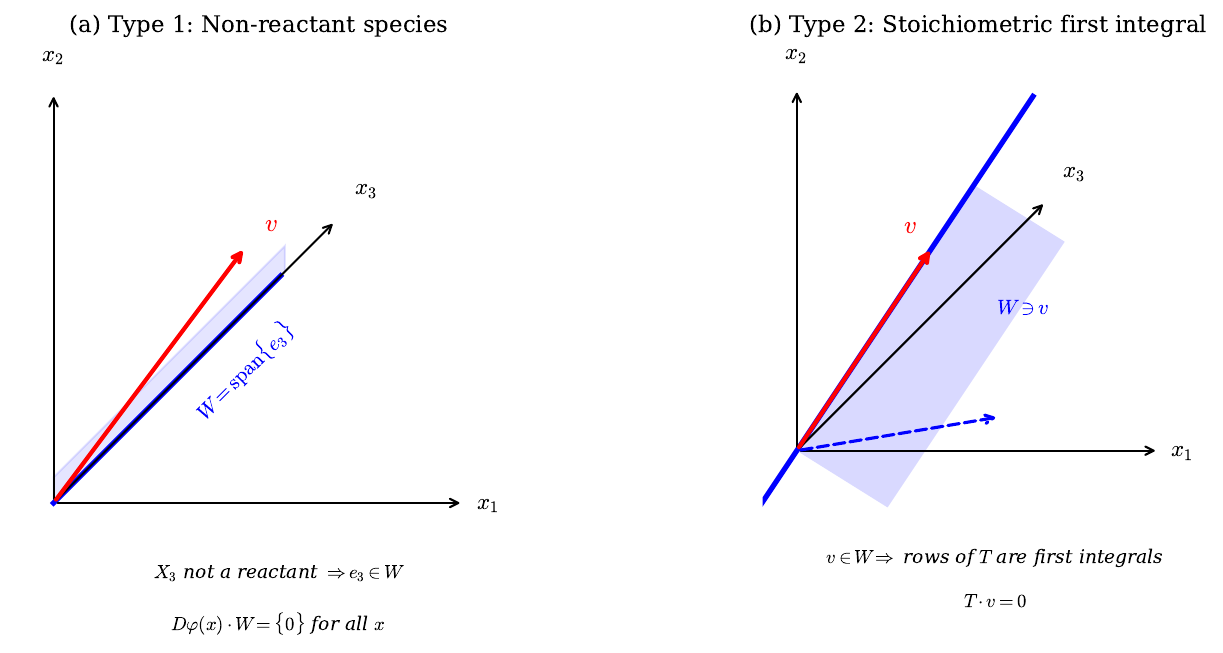}
\caption{Geometric interpretation of Type~1 and Type~2 invariant subspaces 
for a single reaction $X_1+X_2\to X_3$. (a)~Type~1: The subspace $W$ is spanned by 
non-reactant species; here $X_3$ is not a reactant, so $e_3 \in W$. 
The reaction vector $v$ lies outside $W$. The condition 
$D\varphi(x) \cdot W = \{0\}$ holds for all $x$. 
(b)~Type~2: The subspace $W$ contains the reaction vector $v$. 
The rows of the lumping matrix $T$ define stoichiometric first integrals, 
satisfying $T \cdot v = 0$. }
\label{fig:type1_type2}
\end{figure}

\begin{corollary}\label{corall} Let $T$ be a lumping map for system \eqref{singlereac}.
\begin{itemize}
\item Assume that the map corresponds to a subspace $W$ of Type 1; and w.l.o.g. let $X_{p+1},\ldots,X_n$ be the non-reactant species. Then up to Gauss row operations (cf.~Remark \ref{nonuniqrem}) one has ${\rm rank}\,T\geq p$, and 
\[
Tx=\begin{pmatrix} x_1\\ \vdots\\ x_p\end{pmatrix},\quad\text{when  } {\rm rank}\,T=p
\]
or 

\[
Tx=\begin{pmatrix} x_1\\ \vdots\\ x_p\\ \mu_{p+1}(x)\\ \vdots\\ \mu_e(x)\end{pmatrix},\quad\text{when  } {\rm rank}\,T=e>p,
\]

with linear forms $\mu_{p+1},\ldots,\mu_e$ that are subject only to the rank condition. If a lumping map of rank $e>p$ exists, then there also exists a lumping map of rank $p$.
Conversely, every map of the type above is a lumping map. 
\item The linear map $T$ is a lumping map that corresponds to a subspace of Type 2 if, and only if, every row of $T$ is a stoichiometric first integral of system \eqref{singlereac}.
\end{itemize}
\end{corollary}
\begin{proof}
\begin{itemize}
\item For Type 1, $W$ is a subspace of $W_{\rm max}=\sum_{i=p+1}^n \mathbb Re_i$, and consequently $W^\perp$ (a subspace of the dual) contains $\sum_{i=1}^p  \mathbb R x_i$. Obviously, for any subset of non-reactant species, by discarding some equations $\dot x_i=\varphi(x)\nu_i$, with $i\geq p+1$,  one obtains a differential equation system of smaller dimension. 
\item For Type 2, any reducing map $T$ sends $v$ to $0$; in other words, every row of $T$ defines a stoichiometric first integral of \eqref{singlereac}.  
\end{itemize}
\end{proof}
\begin{remark}{\em 
\begin{itemize}
\item Given Type 1, the essential part of the reduction corresponds to the rank $p$ case, with subspace $W_{\rm max}$ and
\[
Tx=\begin{pmatrix} x_1\\ \vdots\\ x_p\end{pmatrix}.
\]
Solving the remaining equations for $x_{p+1},\ldots,x_n$ amounts to quadratures.
\item One can describe $T$ with a  Type 2 subspace in more detail. Assuming (w.l.o.g.) that $\nu_1\not=0$, and setting
\[
T_{\rm max}=\begin{pmatrix} \nu_2&-\nu_1&0&\cdots&0\\
                                              \nu_3&0&-\nu_1&\ddots & \vdots\\
                                            \vdots &\vdots&\ddots&\ddots&0\\
                                               \nu_n&0&\cdots&0&-\nu_1
\end{pmatrix},
\]
one has $T=A\cdot T_{\rm max}$, where $A\in \mathbb R^{e\times (n-1)}$ is of rank $e\leq n-1$. 
\item There may exist subspaces that are both Type 1 and Type 2, although this is the case only for rather special reactions: Let $m_1>0,\ldots,m_p>0$ and consider
\begin{equation*}
m_1X_1+\cdots+m_pX_p  \ce {->[k]} m_1X_1+\cdots+m_pX_p+r_{p+1}X_{p+1}+\cdots+r_nX_n.
\end{equation*}
Thus, for every $j\leq p$ the species $X_j$ is either not involved in the reaction (when $m_j=0$) or its concentration is unchanged by the reaction (so it acts solely as a catalyst). We have 
\[v=\begin{pmatrix} 0\\ \vdots\\0\\r_{p+1}\\ \vdots\\r_n\end{pmatrix}\in \mathbb Re_{p+1}+\cdots+\mathbb Re_n.\]
 One verifies that this is (up to labeling) the only scenario where the subspaces have nontrivial intersection.
\end{itemize}}
\end{remark}

\subsection{Lumpings of reaction networks}
We now consider a mass action network with reactions $\mathcal R_1,\ldots,\mathcal R_d$, and dynamics described by the differential equation
\begin{equation}\label{manyreac}
\dot x=F(x,k):= \sum_{i=1}^d k_i\,\varphi_i(x) v_i,
\end{equation}
with each summand of the form in \eqref{singlereac} and \eqref{singlereacplus}. 
\subsubsection{The genericity condition}
Consider the set $\Pi$ of admissible rate parameters $k=\begin{pmatrix} k_1\\ \vdots\\ k_d\end{pmatrix}$ for the reaction network. (Admissibility depends on the context of the problem: $\Pi$ may contain just one element when the rate constants are known precisely; on the other hand, consideration of $\Pi=\mathbb R_{\geq 0}^d$ means that all reactions with the given graph are considered.) As noted earlier, we call a network {\em generic} if $\Pi$ contains a nonempty open subset\footnote{It would suffice to require this subset to be Zariski-dense.} of $\mathbb R^d$. From an applied perspective, such a condition may reflect that the parameters are known only within some error.\footnote{However, compare the comments on approximate lumpings at the beginning of section \ref{sec:critpar}.} \\
Here we are interested in {\em lumping maps for the generic  network}, which do not depend on the rate parameters.
\begin{proposition}\label{prop:all}
Let $T$ be a linear lumping map for \eqref{manyreac} in a generic setting, for all admissible parameters $k$. Then $T$ is a lumping map for all equations $\dot x=\varphi_i(x)v_i$, $1\leq i\leq d$.
\end{proposition}
\begin{proof} The proof will be given for more general systems that are not necessarily mass action; i.e., linear in $k$. We denote the standard basis of $\mathbb R^d$ by $e_1,\ldots,e_d$.
With $W=\ker T$, by Proposition \ref{HWprop} we have $DF(x,k)W\subseteq W$ for all admissible $k$. Let $\widehat k$ be an interior point of $\Pi$, and $1\leq j\leq d$. Then for all sufficiently small $\varepsilon$, and any $w\in W$ we have 
\[
 DF(x,\widehat k+\varepsilon e_j)w-DF(x,\widehat k)w  \in W.
\]
With
\[
 DF(x,\widehat k+\varepsilon e_j)w-DF(x,\widehat k)w = \left(\varepsilon DF(x,e_j)+o(\varepsilon)\right)w,
\]
with $o(\varepsilon)$ standing for terms of order $>1$, this implies $DF(x,e_j)w\in W$.
\end{proof}
\begin{remark}{\em 
This Proposition also applies to Ovchinnikov et al. \cite{OVPT2021} because in their algorithm the lumping maps under consideration only act on variables.}
\end{remark}
\subsubsection{Characterization of lumping maps}
By Proposition \ref{prop:all}, obtaining a subspace $W$ such that $DF(x,k)W\subseteq W$ for all $x$ and all $k$, means to find a subspace $W$ that is of Type 1, or Type 2, for every reaction in the network. \\
In the settings with all subspaces of Type 2 the rows of $T$ determine common stoichiometric first integrals for all reactions. \\
The other settings require a more detailed investigation.
\begin{proposition}\label{prop:mixlump}
Let $T$ be a linear lumping map for system \eqref{manyreac}, and $W=\ker T$ the corresponding joint invariant subspace for all Jacobians. 
\begin{enumerate}[(a)]
    \item 
 If $W$ is (w.l.o.g) of Type 1 for reactions $\mathcal R_1,\ldots,\mathcal R_{d^*}$, $1\leq d^*\leq d$, with common non-reactant species $X_{p+1},\ldots,X_n$, and of Type 2 for the remaining ones (if any), then ${\rm rank}\,T\geq p$, and one of the following holds (up to modifications by Gauss row operations; see Remark \ref{nonuniqrem}):
\begin{enumerate}[(i)]
\item If $T$ has rank $p$, then
\[
Tx=\begin{pmatrix}x_1\\ \vdots\\ x_p\end{pmatrix}\text{  and  } v_j=\begin{pmatrix} 0\\ \widehat v_j\end{pmatrix},\, \widehat v_j\in \mathbb R^{n-p} \text{  for all  } j>d^*.
\]
Thus the coordinate functions $x_1,\ldots,x_p$ are first integrals for every reaction that corresponds to a Type 2 subspace.
\item If $T$ has rank $e>p$, then
there exist linear forms $\mu_{p+1},\ldots,\mu_e$ such that
\[
Tx=\begin{pmatrix}x_1\\ \vdots\\ x_p\\ \mu_{p+1}\\ \vdots\\ \mu_e\end{pmatrix}\text{  and  } v_j=\begin{pmatrix} 0\\ \widehat v_j\end{pmatrix},\, \widehat v_j\in \mathbb R^{n-p} \text{  for all  } j>d^*,
\]
and furthermore all $\mu_j(v_{\ell})=0$, $p+1\leq j\leq e,\, \ell>d^*$. Thus $x_1,\ldots,x_p,\,\mu_{p+1},\ldots,\mu_e$ are first integrals for every reaction that corresponds to a Type 2 subspace. Conversely the conditions on the $\mu_i$ are sufficient for a lumping map.
\end{enumerate}
\item In scenario (ii) above
\[
\widetilde Tx=\begin{pmatrix}x_1\\ \vdots\\ x_p\end{pmatrix}
\]
defines a lumping map of rank $p$.
\item Moreover in scenario (ii) above, there exist polynomials $g_{p+1},\ldots,g_e$ such that
\[
\frac{d\mu_j}{dt}= g_j(x_1,\ldots,x_p),\quad p+1\leq j\leq e.
\]
\end{enumerate}
\end{proposition}
\begin{proof}
Here $W$ is a subspace of $\mathbb R e_{p+1}+\cdots+\mathbb R e_n$, as seen from Propositions \ref{prop:onereac} and \ref{prop:all}, and the assertion on the form of $T$ in part (a) follows by passing to the dual space. The remaining statements for (a) follow because every row of $T$ determines a stoichiometric first integral for each of the reactions $\mathcal R_j$, $j>d^*$. For parts (b) and (c), note that every summand $k_i\varphi_i(x)v_i$, $1\leq i\leq d^*$ (corresponding to Type 1) depends only on $x_1,\ldots,x_p$.
\end{proof}
 This leads us to a practical (if perhaps disappointing) conclusion.
\begin{corollary}\label{cor:mixlump}
    Let $T$ be a linear lumping map, and $W=\ker T$ the corresponding joint invariant subspace for all Jacobians. If this subspace is of Type 1 for some reaction, and the rank of $T$ is minimal with respect to Proposition \ref{prop:mixlump}, then there exist indices $j_1,\ldots,j_p$ such that 
    \[
Tx=\begin{pmatrix} x_{j_1}\\  \vdots\\ x_{j_p}\end{pmatrix},
\]
up to Gauss row operations.
\end{corollary}
The above proposition and corollary represent new results, and they lead to a great simplification in the characterization of, and search for lumping maps (recalling that pure Type 2 will pose only an elementary linear algebra problem). We will utilize this observation to present an algorithm in the upcoming subsection.

    Generally one may prefer lumpings of minimal rank. But the statements in part (a), scenario (ii) about lumpings of non-minimal rank may be relevant for constrained lumpings.\\
    We note a few observations regarding Proposition \ref{prop:mixlump}.
\begin{remark}{\em 
\begin{enumerate}[(a)]
    \item The simplest scenario appears when only Type 1 subspaces are involved. Then reduction amounts to eliminating common non-reactant species.
    \item We take  a closer look when Type 2 subspaces are also involved for some reactions. As above we may assume that $W\subseteq \mathbb R e_{p+1}+\cdots +\mathbb R e_n$ due to the Type 1 subspaces. Now consider a reaction
    \[
    \sum m_iX_i\rightharpoonup \sum r_iX_i; \quad \text{  with  }v=\begin{pmatrix}
        r_1-m_1\\ \vdots\\ r_p-m_p\\r_{p+1}-m_{p+1}\\ \vdots \\r_n-m_n
    \end{pmatrix}
    \]
    that corresponds to a Type 2 subspace. Then necessarily $r_1=m_1,\ldots,r_p=m_p$ and the reaction has the detailed form
    \[
    \sum_{i=1}^n m_i X_i\rightharpoonup \sum_{i=1}^p m_iX_i+\sum_{i=p+1}^n r_iX_i.
    \]
    Thus, for every $j\leq p$ the species $X_j$ is either not involved in the reaction or it acts solely as a catalyst.
    \item In some instances, there may be a choice between Type 1 and Type 2, and the latter may yield a lumping of smaller rank. As a simple example consider
    \[
    X_1+X_2\rightharpoonup X_3,\quad X_4+X_5\rightharpoonup X_6.
    \]
    For the first reaction take the Type 1 subspace $W_1=\sum_{i\geq 3}\mathbb R e_i$. Then one may choose the Type 1 subspace $W_2=\mathbb R e_6+\sum_{i\leq 3}\mathbb  R e_i$ for the second reaction, and obtain a reduction to dimension four (eliminating the common non-reactant species $X_3$ and $X_6$). But on the other hand, $W_1$ is a Type 2 subspace for the second reaction, and thus one obtains a reduction to dimension two.
\end{enumerate}
}
\end{remark}

\begin{figure}[htbp]
\centering
\includegraphics[width=1\textwidth]{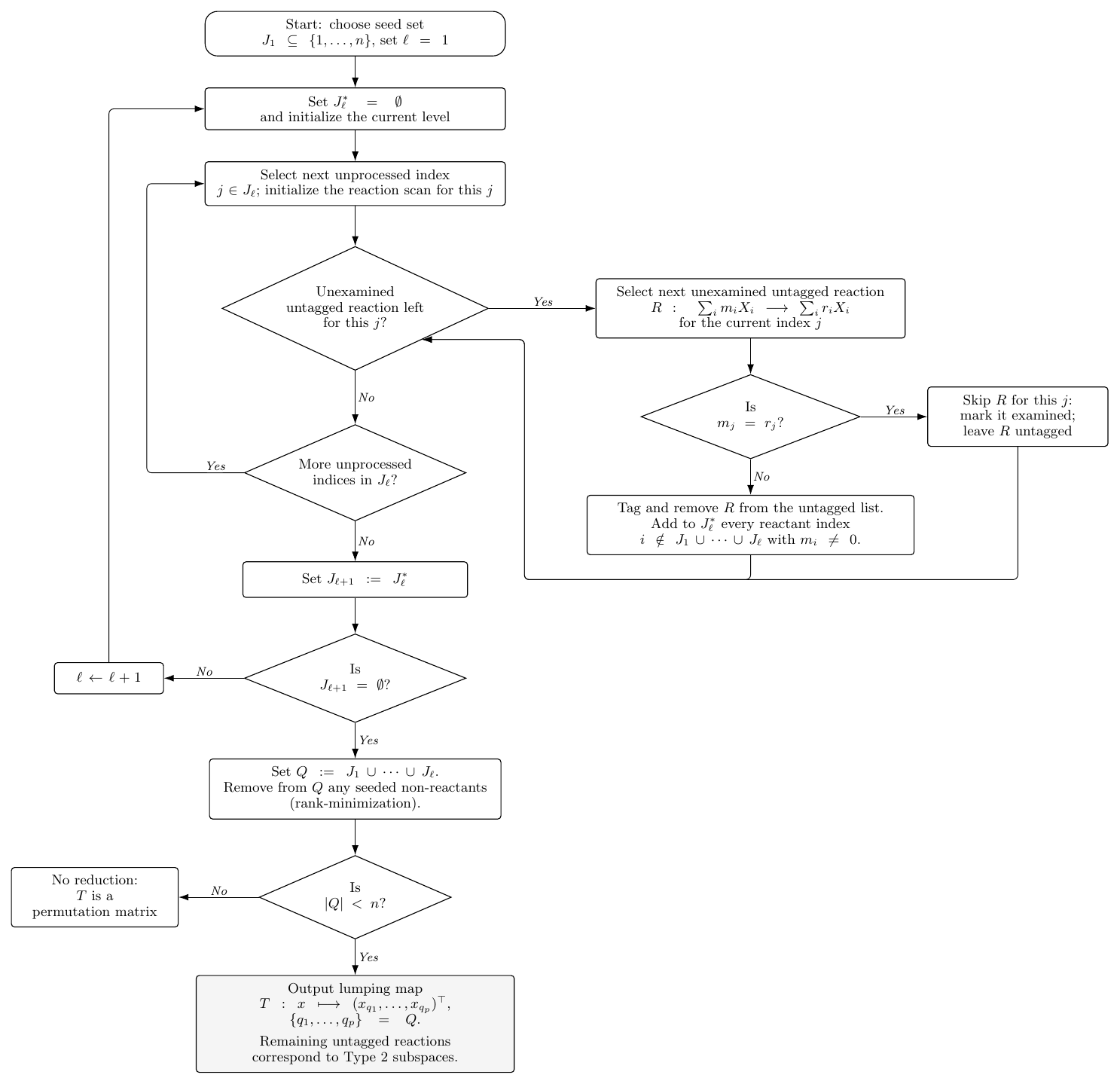}
\caption{\textbf{Construction of a generic linear lumping map from a seed set}
(Algorithm of Section~\ref{sec:genconstr}). Given the reaction network and an
initial set $J_1\subseteq\{1,\ldots,n\}$, the algorithm builds successive index
sets $J_1,J_2,\ldots$ until $J_{\ell+1}=\emptyset$. At each level~$\ell$, the
inner double loop scans, for every $j\in J_\ell$, all reactions still untagged
from previous levels: if $m_j=r_j$ the reaction is left untagged and the scan
moves on (the same reaction may still be processed for a different~$j$);
otherwise the reaction is tagged and its previously-unseen reactant indices are
added to $J_\ell^*$. On termination, $Q=J_1\cup\cdots\cup J_\ell$ (minus seeded
non-reactants that are not stipulated by constraints) indexes the coordinates retained by the projection~$T$; the
tagged reactions correspond to Type~1 subspaces, and the reactions still
untagged at termination correspond to Type~2 subspaces.}
\label{fig:algorithm}
\end{figure}
\subsubsection{Construction of lumping maps}\label{sec:genconstr}
The observations in the previous subsections permit to construct all lumping maps with a given seed set that are of minimal rank that involve a reaction with a Type 1 subspace. (Lumpings that involve only Type 2 subspaces just pose a linear algebra problem.) We will  always  a priori designate a set of species to be included in the lumping; be it by choice or stipulated by constrained lumping. Throughout we will use the correspondence (by duality) between Proposition \ref{HWprop} and Corollary \ref{HWcor}.

Recall from Proposition \ref{prop:all} that a lumping map in the generic setting shares the lumping property for every single reaction. To illustrate the strategy, we fix an index $j$ from the set $J_1$ of the indices of those species designated to appear in the lumping. Now consider first a single reaction
  \[
        \sum m_iX_i\rightharpoonup \sum r_iX_i
        \]
and the corresponding term
\[
\dot x_j=k\, (r_j-m_j) x_1^{m_1}\cdots x_n^{m_n}
\]
in the differential equation.
Then, if $r_j-m_j\not=0$, one sees from the entry $\# j$ of the differential equation that all indices $\ell$ with $m_\ell\not=0$ appear on the right-hand side (see Proposition \ref{prop:onereac} with proof, and Corollary \ref{corall}), so the set of designated indices must be augmented by these. In case $r_j=m_j$ no action is needed. Going through all reactions, and through all $j\in J_1$ we get a new set $J_2\supseteq J_1$. Now proceed, but only the untagged reactions need further consideration, et cetera. 

Proceeding yields
 a (certainly non-optimized) procedure, which is summarized in 
Figure~\ref{fig:algorithm}.\\
Thus we designate a starter set of species with indices in $J_1$ and then, according to Proposition \ref{prop:mixlump} and its proof (and passing to the dual space), recursively augment the present set $J_\ell$ by the reactants for each reaction that involves a species with index in $J_\ell$; individual steps are detailed below. The rank of the lumping map thus obtained may not yet be minimal with regard to a given set of stipulated species and a given set of seeds: this occurs if and only if the starter set $J_1$  contains non-reactants that are not stipulated by constraints. In this case, remove these non-reactants in a final step.
\begin{enumerate}
    \item Start with a nonempty subset $J_1\subseteq\{1,\ldots,n   \}$ (indices of the ``seeded'' species).
    \item For $\ell\geq 1$, given $J_1,\ldots,J_\ell\subseteq\{1,\ldots,n   \}$, set $J_\ell^*=\emptyset$.
    \begin{itemize}
        \item For every $j\in J_\ell$: Consider all yet untagged reactions
        \[
        \sum m_iX_i\rightharpoonup \sum r_iX_i.
        \]
        If $m_j= r_j$ then proceed. (In this case, $X_j$ does not appear in the reaction or acts solely as a catalyst). Otherwise, if $m_j\not=r_j$, then augment $J_\ell^*$ with all $i\not\in J_1\cup\cdots\cup J_\ell$ such that $m_i\not=0$ (i.e., the indices of all reactant species for this reaction). Moreover, mark the reaction with a tag and remove it from further consideration.
        \item Upon completing the run through all $j\in J_\ell$, set $J_{\ell+1}:=J_\ell^*$.
    \end{itemize}
    \item If $J_{\ell+1}\not=\emptyset$, then repeat the above loop with $J_1,\ldots,J_{\ell+1}$ and the remaining untagged reactions.
    \item If $J_{\ell+1}=\emptyset$, then 
    \[
    x\mapsto \begin{pmatrix}
        x_{q_1}\\ \vdots \\x_{q_p}
    \end{pmatrix}, \quad \left\{q_1,\ldots,q_p\right\}=J_1\cup\cdots\cup J_\ell
    \]
    defines a lumping map whenever $p<n$. (The condition $p<n$ is nontrivial; frequently one will end up with just a permutation matrix $T$.)
    \item The remaining untagged reactions correspond to Type 2 subspaces.
    \item As described, the procedure also works for a particular class of constrained lumpings: If certain concentrations $x_{i_m}$ are prescribed in the constraints, then apply the above algorithm with $J_1$ comprising the corresponding indices. (In the case when certain constraints involve linear combinations of more than one species, the above approach may be modified, but the algorithm in Ovchinnikov et al.~\cite{OVPT2021} seems more appropriate.)
\end{enumerate}

\begin{remark}
    Alternatively, if the differential equation system for the network is given in the explicit form 
    \[
    \dot x_i=F_i(x,k),\quad 1\leq i\leq n,
    \]
    then one may replace part 2 of the procedure by the following: \\
For every $j\in J_\ell$, augment $J_\ell^*$ by all indices $s\not\in J_1\cup\cdots\cup J_\ell$ for which $x_s$ appears in $F_j$.
\end{remark}

 We illustrate this approach via Proposition \ref{prop:mixlump}:
\begin{example}\emph{Consider the following simple network:
    \[
    X_1+X_2\rightharpoonup X_3,\quad X_4+X_5\rightharpoonup X_6.
    \]
    \begin{itemize}
        \item Choosing $J_1=\{1\}$, in step 2 one obtains $J_1^* = \{2\}$ for the first reaction, thus $J_2=\{2\}$, and the second reaction requires no action. Proceeding with $J_2$ and the (solely remaining) second reaction, one sees that no further action is required. Thus one gets a lumping map with $Tx=\begin{pmatrix}
            x_1\\x_2
        \end{pmatrix}$.
        The (untagged) second reaction corresponds to a Type 2 subspace for the first.
        \item On the other hand, starting with  $J_1=\{3\}$, step 2 will yield $J_1^*=\{1,\,2\}$, and we end up with a lumping map $Tx=\begin{pmatrix}
            x_1\\x_2\\x_3
        \end{pmatrix}$, and the non-reactant species $X_3$ can be removed at the end.
    \end{itemize}}
\end{example}

As noted in the Introduction, there exist algorithms to compute lumping maps in the generic setting, notably Ovchinnikov et al.~\cite{OVPT2021}. However, by our results one only needs to consider coordinate functions, which greatly reduces the necessary computational effort.

\begin{example}{\em
Consider the example in Ovchinnikov et al.~\cite{OVPT2021}, subsection 4.1 and Figure 1. The task is to find (in the notation of \cite{OVPT2021}) a constrained lumping that preserves $AI2i$. This can be achieved by direct inspection: Looking at all three reactions involving $AI2i$ (bottom of part A of the figure), one sees that the reactants $AI2e$ and $DPD$ must be included. Then, looking at the remaining reaction involving $DPD$, one finds that the reactant $SHR$ must be included, and so on. In this way one proceeds (``going up the middle branch'' of part A in this figure, and including all species in complexes from which an arrow points to one that is already included) to find that $SAH$ must be included. Now consider the remaining reactions involving $SAM$ in this clipping of the reaction scheme:
\[
\begin{array}{ccc}
   &  &  Met\\
     & & \downarrow\\
   SAM_{Decarb}&\leftarrow &SAM\\
   & & \downarrow\\
   & & SAH
\end{array}
\]
One sees that the reactant $SAM$ must be included, as well as the reactant $Met$ in the next step, but the non-reactant $SAM_{Decarb}$ need not. Finally, for the topmost reaction in part A the reactant $Nut$ must be included\footnote{The intuitive strategy to ``walk along the graph of the reaction network'', which we used here, could obviously be cast in a more formal manner.}. We thus arrive at the same reduction as given in \cite{OVPT2021}.}
\end{example}

\subsection{A look at some other types of kinetics}
The discussion above was restricted to mass action kinetics, because mass action systems form an important class, and it was possible to obtain a complete overview of all lumping maps. To explore whether other common kinetics might offer a greater variety of linear lumping maps, we discuss some more general settings here.
\subsubsection{A single reaction with product form velocity}
We keep the stoichiometry conditions but different types of reaction velocities are permitted here. Thus, we consider a reaction equation of the form 
\begin{equation}\label{generalreac}
\dot x=F(x,b)=\varphi(x,b) v
\end{equation}
with $v$ having integer entries, and $\varphi$ analytic in $x$ and in parameters $b$. 
A large and relevant class of reactions assumes a product form
\begin{equation}\label{eq:prodform}
    \varphi(x,b)=\varphi_1(x_1,\beta_1)\cdots\varphi_n(x_n,\beta_n)
\end{equation}
for the reaction velocities. In addition to mass action, this class includes Michaelis--Menten or Hill terms.\\
Proposition \ref{HWprop} still holds for system \eqref{generalreac} in general, and for system \eqref{eq:prodform} in particular. Therefore, we search for subspaces $W$ that are invariant with respect to all Jacobians. We again find two types, with $DF(x,b)=v\, D\varphi(x,b)$, and using  arguments analogous to the proof of Proposition \ref{prop:onereac}:
\begin{itemize}
\item Type 1: One has $D\varphi(x,b)\,W=\{0\}$ for all $x$ and $b$.
\item Type 2: One has $v\in W$.
\end{itemize}
The second type yields stoichiometric first integrals, as before. Considering Type 1 for product form, one finds
\[
D\varphi(x,b)=\varphi(x,b)\cdot \begin{pmatrix}\dfrac{\varphi_1'(x_1,\beta_1)}{\varphi_1(x_1,\beta_1)}& \cdots &\dfrac{\varphi_n'(x_n,\beta_n)}{\varphi_n(x_n,\beta_n)}\end{pmatrix},
\]
with the prime denoting the derivative with respect to $x_i$, respectively.
Now the same argument as in the mass action case yields
\[
D\varphi(x,b)\, w=0 \text{  for all  } x \Longleftrightarrow \dfrac{\varphi_1'}{\varphi_1}\cdot w_1=\cdots =\dfrac{\varphi_n'}{\varphi_n}\cdot w_n=0 \text{  for all  } x ,
\]
where $\varphi_j'/\varphi_j$ depends only on $x_j$, and is non-constant unless $\varphi_j=c\cdot \exp(ax_j)$. So, if one excludes the (chemically not very relevant) case where two or more non-constant $\varphi_j$ are of this form, the argument again shows that $W$ corresponds to non-reactant species. We conclude that Type 1 also yields the same conditions as mass action kinetics.
\subsubsection{Generic networks of product form velocity reactions}
We restrict attention to reaction equations of the type
\[
\dot x=F(x,b,k)=\sum_{i=1}^d k_i \varphi_i(x,b)\,v_i
\]
with additional parameters $k_i$. One readily sees that Proposition \ref{prop:all} remains valid in this more general setting. Thus, a linear lumping map for the network will be a lumping map for every single reaction, and a case-by-case discussion is again possible. With all reaction equations having product form, Proposition \ref{prop:mixlump} and Corollary \ref{cor:mixlump} apply almost verbatim. \\

To construct all lumping maps for a given generic network, one may imitate the pattern for mass action systems, employing the alternative approach for the loop.

\subsection{An interim conclusion and a comparison with computational methods}
As we have shown, the range of possible linear lumping maps for generic mass action (and all relevant
product-form) reaction networks is rather limited. The reason is that the range for single 
reactions is rather limited, and the genericity condition forces lumpability for every single reaction.

This result provides theoretical context for computational lumping methods such as 
CLUE~\cite{OVPT2021} and ERODE~\cite{CardelliTTV2017}. These algorithms seek \emph{structural} 
(parameter-independent) lumpings---included in the generic case analyzed here. Our characterization explains why such methods often find only modest reductions or none at all: for generic 
parameters, only Type 1 (non-reactant species elimination) and Type 2 (stoichiometric first 
integrals) lumpings exist.

This observation focuses interest on non-generic parameter regimes. In applications, rate 
constants are not arbitrary but arise from physical and chemical considerations. Special 
relationships among parameters---such as detailed balance, microscopic reversibility, or 
enzyme saturation conditions---may place systems at or near critical parameter values where 
non-trivial lumpings become available. In section 4 we will develop systematic methods to identify 
such critical parameters.


\section{Critical parameters}\label{sec:critpar}

In the present section we will focus on non-generic reaction networks, and exclusively deal with mass action systems.\\ We return to a notion that was mentioned earlier, now in a formal manner.
\begin{definition}
    Let system \eqref{parode} be given, and let $e<n$ and $T:\, \mathbb R^n\to\mathbb R^e$ be linear, of full rank. Then we call $k^*$ a critical parameter value for $T$ if $T$ is solution-preserving from $\dot x=F(x,\,k^*)$ to some polynomial system $\dot y=G(y,k^*)$ on $\mathbb R^e$.
\end{definition}
The focus on critical parameters is not new: The classical work by Wei and Kuo \cite{WK1969a, WK1969b} on first order networks actually deals with critical parameters: If a mass action network comprises only monomolecular reactions  $X_j\rightarrow X_i$ between species $X_1,\ldots,X_n$ with rate constants $k_{ij}$, then the dynamics is determined by a differential equation
\begin{equation}
\dot x = A(k)\,x;
\end{equation}
in other words, by the Laplacian.
As noted by Wei and Kuo \cite{WK1969a,WK1969b} (building on Wei and Prater \cite{WP1962}), for linear systems one may generally construct lumping maps via sums of eigenspaces, or generalized eigenspaces. One may extend this to joint eigenspaces of Jacobians in nonlinear settings.
 But there are limitations to this method; for instance eigenspaces cannot be determined exactly in general. In the following we will therefore consider different approaches.\\

In some applications, there is interest in prescribed candidates for linear lumping maps; for instance, these may be motivated by chemical intuition. Since prescribing the lumping map will impose conditions on the rate parameters, we also have a critical parameter problem here.\\

Critical parameters are relevant from a different perspective: While one should not necessarily expect exact lumping for a given system, the system parameters may be close to critical parameters.
Then, loosely speaking, solutions of \eqref{parode} with $k=\widehat k$ remain close to solutions of \eqref{parodes} as long as $\widehat k$ is close to a critical parameter $k^*$. This leads to approximate lumpings, which are well established in the literature; see e.g.\ Wei and Kuo \cite{WK1969b}, Li and Rabitz \cite{LR1990}, and Leguizamon-Robayo et al.\ \cite{LJTTV}.

\subsection{Conditions for critical parameter values}
Proposition \ref{HWprop} opens a path for finding critical parameter values via necessary and sufficient conditions. The following restatement was already established by Li and Rabitz \cite{LR1989}; see their equation (17).
\begin{lemma}\label{lalem}
With $T\in\mathbb R^{e\times n}$ given, let $b_1,\ldots,b_{n-e}\in\mathbb R^n$ be a basis of $\ker T$, and let $B$ be the matrix with columns $b_i$. Then $k^*$ is a critical parameter for $T$ and \eqref{parode} if and only if 
\begin{equation}\label{linparcond}
T\,DF(x,k^*)\,B=0
\end{equation}
for all $x$.
\end{lemma} 
Now let $T$ and (an appropriate choice of) $B$ be fixed. Then,
given a representation of $F$ as a linear combination of vector-valued monomials with the rate parameters $k$ as coefficients, the entries of $k^*$ satisfy a homogeneous linear system of equations. 
Specifically, we can write
\begin{equation}\label{eq:multiindex}
DF(x,k) = \sum_{\alpha} M_\alpha(k) \cdot x^\alpha
\end{equation}
where $M_\alpha(k)$ are matrices depending linearly on $k$ and $x^\alpha$ are monomials,
and $T\,DF(x,k^*)\,B=0$ if and only if $T \cdot M_\alpha(k^*) \cdot B = 0$ for all $\alpha$. This system of equations is linear in the entries of $T$, in the entries of $B$, and in $k$.
There remains the question how to obtain $T$. There are various perspectives to this, and we will discuss two of them.

\subsection{Geometric interpretation}
The set of critical parameters for a given lumping map $T$ forms a \emph{semi-algebraic set} in 
parameter space. $\mathbb{R}^d$.
For fixed $T$ (and $B$), the equations $T M_\alpha(k)B=0$ are homogeneous linear equations in the rate parameters $k$.  If $T$ is also unknown and is written in a row-echelon chart, the resulting equations are at most quadratic in the chart variables and linear in $k$.
In addition, the rate parameters must satisfy positivity conditions.

In the linear setting, the critical variety $V_T \subset \mathbb{R}^d$ consists of one irreducible component. Its \emph{codimension} provides a first measure how ``exceptional'' 
the corresponding parameter regime is, but note that the physically relevant part is just the intersection $V_T\cap \mathbb R^d_{\geq 0}$.
\begin{itemize}
    \item {\bf Codimension 0 (open set):} The lumping holds generically---this leads only 
    to the ``obvious'' reductions characterized in Section 3. While this is not a scenario we are interested in the present section, it would be impractical to exclude it by modifying the definition, in view of the algorithmic approach obtained below.
    \item {\bf Codimension 1 (intersection of open set and hyperplane):} A single linear relation among parameters, together with the inequalities $k_i\geq 0$, 
    enables the lumping.
    \item {\bf Higher codimension:} Multiple independent parameter constraints are required, together with the inequalities.
\end{itemize}

For applications, low-codimension components are most relevant, as they are most likely to be approximately satisfied by experimentally 
determined parameters.

This geometric perspective connects to \emph{computational algebraic geometry}. Tools such 
as Gr\"obner bases~\cite{CLOS} can decompose the critical variety into irreducible 
components, determine their dimensions, and test membership. For systems with many 
parameters, numerical algebraic geometry methods may be more practical.

\subsection{An algorithmic approach}
While the notion of a critical parameter (apart from the name, perhaps) and criteria were known before, the following results about their computation are new.
 We will show that obtaining all critical parameters and corresponding lumping maps amounts to solving finitely many systems of polynomial equations, up to modifications permitted by Remark \ref{nonuniqrem}. In fact, these modifications open a path toward their determination. We recall a fact from elementary linear algebra.
\begin{remark}\label{rrerem}{\em Row-echelon form (a reminder):
\begin{itemize}
    \item Let $1\leq e<n$ and $T\in\mathbb R^{e\times n}$ of full rank. Then there is an invertible $Q\in\mathbb R^{e\times e}$ such that $QT$ is in reduced row echelon form. Thus, up to a permutation of columns,
    \begin{equation}\label{rowech}
         QT=\begin{pmatrix}
        E_e & \widehat T
    \end{pmatrix},
    \end{equation}
    where $E_e $ denotes the $e\times e$  identity matrix, and $\widehat T\in\mathbb R^{e\times(n-e)}$. Then the columns of the matrix 
    \[
    B=\begin{pmatrix}
        \widehat T\\ -E_{n-e}
    \end{pmatrix}
    \]
    form a basis of $\ker (QT)$. (Note that $B$ is not unique; it can be modified with elementary column operations. But it seems that such operations provide no further simplification.)
    \item By the above, it suffices to consider matrices of the form \eqref{rowech} whenever the first $e$ columns of $T$ are linearly independent. In order to cover all possibilities, it suffices to check all subsets of $\{1,\ldots,n\}$ with $e$ elements; thus a total of $\begin{pmatrix}
        n\\e
    \end{pmatrix}$ cases, and renumber variables accordingly.
    \item We do not aim to further refine, or optimize, the procedure sketched above. But we note that in the special case that $T$ contains $r$ zero columns ($1\leq r<n-e$), one may use a special representation with 
    \[
    \widehat T=\begin{pmatrix}
        \widehat T^*& 0
    \end{pmatrix}, \quad \widehat T^*\in \mathbb R^{e\times (n-e-r)},
    \]
    and fewer nonzero entries in $\widehat T$. This may help when feasibility problems emerge in computations.
\end{itemize}
}
\end{remark}
With Lemma \ref{lalem} and the subsequent observations we find (noting en passant that similar observations will hold for more general systems):
\begin{proposition}\label{algocrit} Let a mass action system \eqref{parode} be given.
\begin{enumerate}[(a)]
    \item  Let $1\leq e<n$ and assume that $T$ is in reduced row echelon form \eqref{rowech}. Then the entries of $\widehat T$ and the corresponding critical parameter values are determined by a system of polynomial equations. This system may be written as a system of linear equations for $k^*$, the matrix coefficients being of degree $\leq 2$ in the entries of $\widehat T$.
    \item All linear lumping maps and their corresponding critical parameter values can be obtained from finitely many systems of the type given above.
\end{enumerate}
\end{proposition}
\begin{proof}
    Part (a) is a direct consequence of the Lemma and the subsequent observations. As to part (b), for given rank $e$ the setting of (a) holds, up to choosing a set of linearly independent columns. And for given $n$ there are only finitely many possible ranks of lumping maps.
    \end{proof}
    Of course, feasibility may pose serious obstacles. But still we have reduced the problem to a problem of solving finitely many polynomial systems.
\begin{remark}\label{remconstr}
{\em   A straightforward modification of the procedure works when some rows of $T$ are prescribed; thus one considers critical parameters in the setting of constrained lumping.
Starting with
\[
T=\begin{pmatrix}
    T_1\\ T_2
\end{pmatrix}, \quad T_1\in\mathbb R^{e_1\times n}, \,T_2\in \mathbb R^{(e-e_1)\times n},
\]
with the first $e_1$ rows prescribed, write
\[
T=\begin{pmatrix}
    T_{11}&T_{12}&T_{13}\\ 
     T_{21}&T_{22}&T_{23}
\end{pmatrix},\quad T_{11}\in\mathbb R^{e_1\times e_1},\quad T_{12}\in \mathbb R^{e_1\times(e-e_1)}
\]
in block form, with the remaining blocks of appropriate sizes; in particular $T_{22}\in\mathbb R^{(e-e_1)\times(e-e_1)}$. Up to column permutations, we may assume that $T_{11}$ is invertible, and therefore, for some $Q_1$,
\[
\begin{pmatrix}
    Q_1& 0\\ 0& E
\end{pmatrix}\cdot T= \begin{pmatrix}
    E&\widehat T_{12}&\widehat T_{13} \\
     T_{21}&T_{22}&T_{23}
\end{pmatrix}=:T^*.
\]
Next,
\[
\begin{pmatrix}
    E&0\\ - T_{21} &E
\end{pmatrix}\cdot T^*=\begin{pmatrix}
    E&\widehat T_{12}&\widehat T_{13} \\
     0&\widetilde T_{22}&\widetilde T_{23}
\end{pmatrix}=T^{**},
\]
and we may assume (up to column permutations) that $\widetilde T_{22}$ is invertible. Finally there exists $Q_2$ such that
\[
\begin{pmatrix}
    E&0\\ 0& Q_2
\end{pmatrix} T^{**}=\begin{pmatrix}
      E&\widehat T_{12}&\widehat T_{13} \\
     0& E &\widehat T_{23}
\end{pmatrix}=:\widehat T.
\]
Note that the first block row of $\widehat T$ contains only constant matrices, representing (modified) constraints. The entries of $\widehat T_{23}$ may be chosen freely. \\
Now proceed as above with
\[
\widehat B=\begin{pmatrix}
     \widehat T_{12}\widehat T_{23}-\widehat T_{13}\\  -\widehat T_{23}\\ E
\end{pmatrix}.
\]
}
\end{remark}
\subsection{Examples}
    We discuss some examples to illustrate the procedure.
\begin{example}{\em 
For a first order reaction network with three species, and no reaction between $X_1$ and $X_3$ we obtain the  Laplacian matrix
\[
A(k)=\begin{pmatrix}-k_1&k_{-1}&0\\ k_1&-(k_{-1}+k_2)& k_{-2}\\ 0 & k_2 & -k_{-2}\end{pmatrix}.
\]
To determine lumping maps of rank two we make the ansatz
\[
T=\begin{pmatrix} 1 & 0 & t_1\\ 0&1&t_2\end{pmatrix}
\]
with parameters $t_1,\,t_2$. This yields
\[
B=\begin{pmatrix} t_1\\ t_2\\ -1\end{pmatrix},
\]
and with Lemma \ref{lalem} we have
\[
TA(k)B=\begin{pmatrix} -t_1k_1+t_2k_{-1}+t_1t_2k_2+t_1k_{-2}\\ t_1k_1-t_2k_{-1}+(t_2-1)t_2k_2+(t_2-1)k_{-2}  \end{pmatrix}.
\]
The condition $TA(k)B=0$ may be rewritten in the form
\[
\begin{pmatrix} -t_1&t_2&t_1t_2&t_1\\ t_1&-t_2&(t_2-1)t_2&t_2-1  \end{pmatrix}\begin{pmatrix}
    k_1\\ k_{-1}\\ k_2\\ k_{-2}
\end{pmatrix}=0.
\]
This linear system in the $k_i$ admits a nontrivial solution only if the matrix has rank one (rank zero being impossible). In turn, this is equivalent to $t_1=t_2=0$ or $t_1+t_2=1$.
For instance, any choice of $t_1$ (thus $t_2=1-t_1$) will yield a solution to the linear system for the $k_i$. In addition the $k_i$ should be nonnegative, which places restrictions on $t_1$.\\
As for one specific example, consider $t_1=1,\,t_2=0$. Then the condition on the parameters reads $k_{-2}=k_1$, and one verifies that $Tx=\begin{pmatrix}
    x_1+x_3\\ x_2
\end{pmatrix}$ yields a solution preserving map to a system in dimension two.
}
\end{example}
\begin{example}\label{ex:mmsmall}{\em 
    The reversible Michaelis--Menten system represents a well-known model for an enzyme-catalyzed reaction, with species $S$ (substrate), $E$ (enzyme), $C$ (complex) and $P$ (product). The reactions are
\begin{equation*}
   \ce{$S$ + $E$ <=>[$k_1$][$k_{-1}$] $C$ <=>[$k_2$][$k_{-2}$] $E$ + $P$ }.
\end{equation*}
With mass-action kinetics one obtains the differential equation system
\begin{align*} 
\dot s &= -k_1es + k_{-1}c, \\
\dot e &= -k_1es + (k_{-1}+k_2)c- k_{-2}ep,\\
\dot c &=  k_1es - (k_{-1}+k_2)c+ k_{-2}ep,\\
\dot p &= k_2c - k_{-2}ep.
\end{align*}
We search for a reduction to dimension one, thus a lumping matrix $T$ of rank one, with the ansatz
\[
T=\begin{pmatrix} 1&t_1&t_2&t_3\end{pmatrix},
\]
thus
\[
B=\begin{pmatrix}
    t_1&t_2&t_3\\ -1&0&0\\ 0&-1&0\\ 0&0&-1
\end{pmatrix}.
\]
With
\[
DF=\begin{pmatrix}-k_1e&-k_1s&k_{-1} & 0\\
-k_1e&-k_1s-k_{-2}p&k_{-1}+k_2 & -k_{-2} e\\
k_1e&k_1s+k_{-2}p&-k_{-1}-k_2 & k_{-2} e\\
0&-k_{-2}p&k_2 & -k_{-2} e\\
\end{pmatrix}
\]
one finds
\[
T\cdot DF\cdot B=\begin{pmatrix}
   - t_1\cdot k_1e(1+t_1-t_2)+ k_1s(1+t_1-t_2)+k_{-2}p(t_1-t_2+t_3)\\
  -t_2 \cdot k_1e(1+t_1-t_2)-(1+t_1-t_2)k_{-1}-(t_1-t_2+t_3)k_2\\
  - t_3\cdot k_1e(1+t_1-t_2)+(t_1-t_2+t_3)\cdot k_{-2}e
\end{pmatrix}^{\rm tr}.
\]
After some obvious simplifications, we arrive at
\[
\begin{pmatrix}
    1+t_1-t_2&0&0&0\\
    0&0&0&t_1-t_2+t_3\\
    0&t_1-t_2+t_3&1+t_1-t_2&0
\end{pmatrix}\cdot\begin{pmatrix}
    k_1\\ k_2\\ k_{-1}\\ k_{-2}
\end{pmatrix} =0.
\]
This system always admits nontrivial solutions. But if one requires all parameters $k_i$ to be nonzero (thus all reactions to be involved), then necessarily $1+t_1-t_2=t_1-t_2+t_3=0$; equivalently
\[
t_2=1+t_1\text{  and  } t_3=1.
\]
One sees that $T$ defines a (stoichiometric) first integral of the system.
}
\end{example}
\begin{figure}[htbp]
\centering
\includegraphics[width=0.75\textwidth]{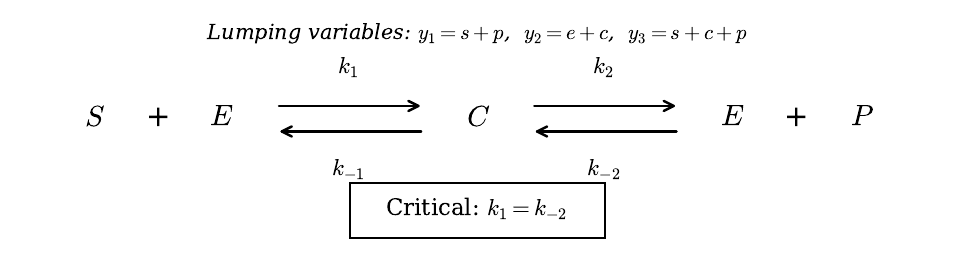}
\caption{The reversible Michaelis--Menten reaction network. Substrate $S$ 
and enzyme $E$ form complex $C$, which can dissociate to release product 
$P$ and regenerate the enzyme. The critical parameter condition 
$k_1 = k_{-2}$ enables a constrained lumping that preserves the two
stoichiometric first integrals and closes the dynamics of
$y_1 = s + p$. See Example~\ref{ex:mm} for the detailed analysis.}
\label{fig:michaelis_menten}
\end{figure}

\begin{example}\label{ex:mm}
{\em
We look at the Michaelis--Menten system again, to illustrate Remark \ref{remconstr}. The essential new aspect is that the first two rows of $T\in\mathbb R^{3\times 4}$ are prescribed; here we choose them as stoichiometric first integrals. Considering the case when the first three columns of $T$ are linearly independent, we have with Remark~\ref{remconstr}:
\[
T=\begin{pmatrix} 1&0&1&1\\ 0&1&1&0\\0&0&1&t\end{pmatrix}, \quad B=\begin{pmatrix}t-1\\ t\\ -t\\ 1\end{pmatrix}.
\]
The computation of $T\cdot DF\cdot B$ is straightforward. Due to the prescribed first integrals, its first and second entries are zero, and from the third entry one obtains the condition
\[
k_1e(t-1)+(k_1s+k_{-2}p)t+(k_{-1}+k_2)t+k_{-2}e-k_{-2}pt^2-k_2t^2-k_{-2}et=0,
\]
which must hold for all choices of $s,\,e,\,c,\,p$. We obtain the system
\[
\begin{pmatrix}
    0&t(1-t)&t&0\\
    t-1&0&0&1-t\\
    t&0&0&0\\
    0&0&0&t(1-t)
\end{pmatrix}\cdot\begin{pmatrix}
    k_1\\ k_2\\ k_{-1}\\ k_{-2}
\end{pmatrix} =0.
\]
The matrix has rank three whenever $t\not\in\{0,1\}$, thus nontrivial solutions always exist (but note the requirement $k_i\geq 0$). For $t=1$ the rank equals two, and the parameter condition is $k_1=k_{-1}=0$; in other words, one reversible reaction pair is discarded. In the distinguished case $t=0$ the rank equals one, and 
 there is only one parameter condition left, viz.\ $k_1=k_{-2}$, and lumping with coordinates
$y_1=s+p$, $y_2=e+c$, and $y_3=s+p+c$ (involving a column transformation of $T$ yields the reduced system

\[
\dot y_1=-k_1(y_1+y_2-y_3)y_1+(k_{-1}+k_2)(y_3-y_1),
\qquad
\dot y_2=0,
\qquad
\dot y_3=0.
\]

 }
\end{example}
\section{Proper lumping and related concepts}



\subsection{Criteria for proper lumping}\label{subsec:criteria}
Wei and Kuo \cite{WK1969a} discuss a classical example for a prescribed lumping map and ensuing conditions for critical parameter values. Following them, we call a lumping {\em proper} if there exist a partition
\[
\left\{1,\ldots,n\right\} =I_1\dot\cup\cdots\dot\cup I_r
\]
 and positive constants $\gamma_1,\ldots,\gamma_n$ such that the lumping is given by
\begin{equation}\label{properlumpmatgam}
\widetilde Tx=\begin{pmatrix} \sum_{j\in I_1} \gamma_j x_j \\ \vdots \\ \sum_{j\in I_r}\gamma_j x_j\end{pmatrix};
\end{equation}
thus every species concentration appears in exactly one entry of $\widetilde T$. We will refer to the $I_p$ as {\em blocks}. \\
From a mathematical perspective the conditions on $\widetilde T$ can be simplified: By a linear coordinate transformation with matrix ${\rm diag}\, (\gamma_1,\ldots,\gamma_n)$ (a scaling that sends each  $x_j$ to $\gamma_jx_j$) one obtains a system with lumping map 
\begin{equation}\label{properlumpmat}
 Tx=\begin{pmatrix} \sum_{j\in I_1} x_j \\ \vdots \\ \sum_{j\in I_r} x_j\end{pmatrix}.
\end{equation}
Thus row $\# \ell$ of $T$ contains only entries $1$ and $0$, and the entry equals $1$ if and only if the column index lies in $I_\ell$. Compare also the notion of forward differential equivalence (FDE) in Cardelli et al. \cite{CardelliTTV2017}. We will use this simplification in the following, to keep notation at bay, but note Remark \ref{scalerem} below.\\ 
The following result was stated by Wei and Kuo for first order reactions. The general  version of the statement, and our proof using Lemma \ref{lalem}, seem to be new. For a different characterization see Cardelli et al. \cite{CardelliTTV2017}.
\begin{proposition}\label{colslem} Let system \eqref{parode} be given, and abbreviate $M=DF(x,k^*)$. With $T$ as in \eqref{properlumpmat}, and $1\leq p,q,\leq r$, denote by  $M_{pq}$  the $|I_p|\times |I_q|$ submatrix obtained from $M$ by deleting all rows with numbers not in $I_p$ and all columns with numbers not in $I_q$. \\
Then  $k^*$ is a critical parameter for $T$ and \eqref{parode} if, and only if, for each pair $(p,\, q)$ all column sums of $M_{pq}$ are equal (as functions of $x$).
\end{proposition}
\begin{proof} 
Let $B$ be such that its columns form a basis of $\ker T$. By Lemma \ref{lalem} it suffices to show that $M$ satisfies $T\cdot M\cdot B=0$ if and only if  all column sums of $M_{ij}$ are equal.\\
One may assume that 
\[
T=\begin{pmatrix} 1&\cdots&1&0&\cdots& 0& 0&\cdots\\
0&\cdots &0&1&\cdots&1&0&\cdots\\
\vdots&&&&&&&\ddots
\end{pmatrix},
\]
and that all rows with a single entry $1$ are gathered in the last columns. Letting
\[
M=\begin{pmatrix} M_{11}&\cdots&M_{1r}\\
          \vdots & \ddots&\vdots \\
               M_{r1}&\cdots & M_{rr}\end{pmatrix}
\]
according to the partitioning, one finds

\[
T\cdot M=\begin{pmatrix}{\rm cols}\,( M_{11})&\cdots&{\rm cols}\,(M_{1r})\\
          \vdots & \ddots&\vdots \\
               {\rm cols}\,(M_{r1})&\cdots & {\rm cols}\,(M_{rr})\end{pmatrix},
\]
where ${\rm cols}\,(M_{pq})$ denotes the row which has as entry $\#\ell$ the sum of the elements of column $\#\ell$ of $M_{pq}$.
Now (in a variant of equation \ref{rowech}) the matrix $B$ built from basis elements of $\ker T$ can be obtained from
\[
\widetilde B=\begin{pmatrix} B_1&0&\cdots& 0\\
                             0& \ddots &  & \vdots\\
                             \vdots& & \ddots &0\\
                              0& \cdots &  0& B_s\\
\end{pmatrix},
\]

with each
\[
B_j=\begin{pmatrix} 1&0&0&\cdots& 0\\
                               -1&1&0&\cdots &0\\
0&-1& 1 && 0  \\
\vdots& & \ddots &\ddots & \vdots\\
0&\cdots &\cdots & -1& 1 \\
0&\cdots &&0 &-1
\end{pmatrix}
\]
of appropriate size, corresponding to a row of $T$ with more than one entry $1$, with a row of zeros appended for every row of $T$ that has just one entry $1$. Now multiplication of $TM$ by $B$ shows, for each index pair $(p,\,q)$, that all entries of ${\rm cols}\,(M_{pq})$ are equal.
\end{proof}

Figure~\ref{fig:proper_lumping} summarizes the construction procedure.
\begin{figure}[htbp]
\centering
\includegraphics[width=0.95\textwidth]{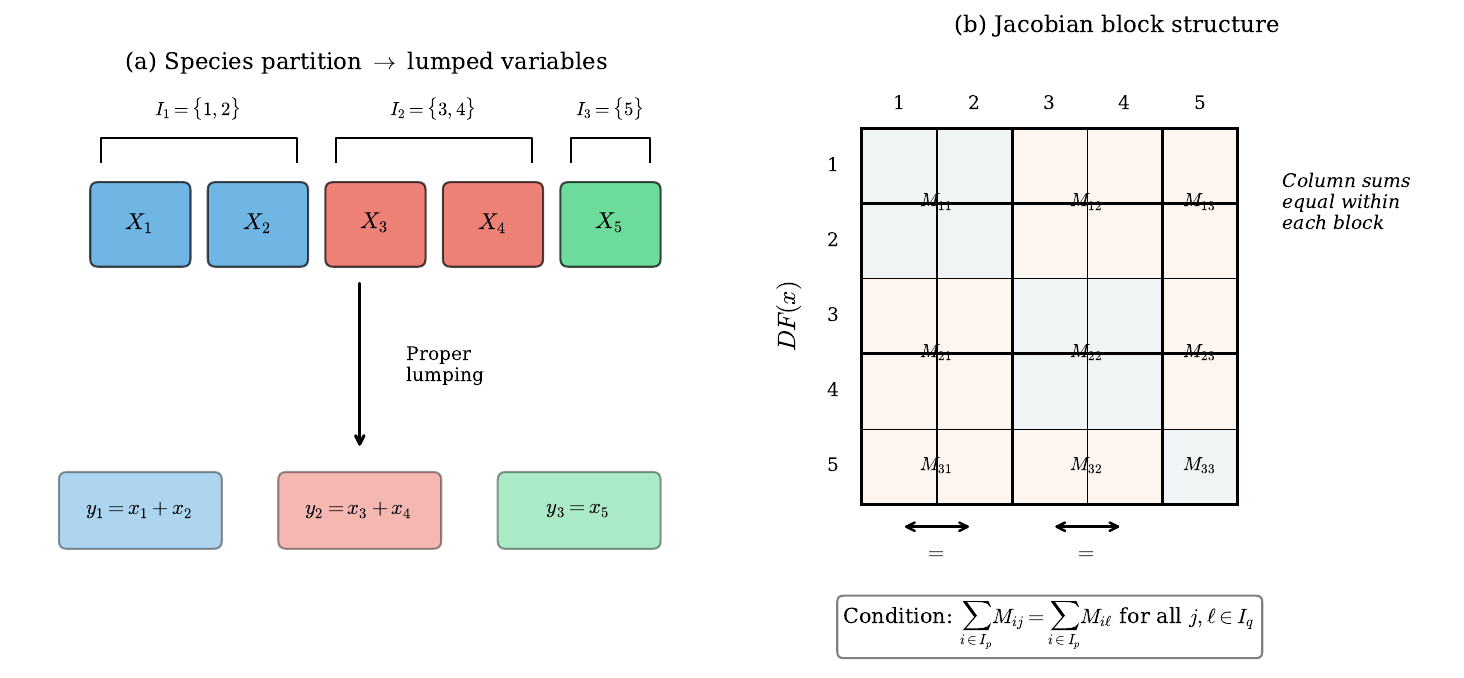}
\caption{Proper lumping and the Jacobian block structure. 
(a)~A partition of species into blocks $I_1, I_2, I_3$ induces lumped 
variables $y_p = \sum_{i \in I_p} x_i$. (b)~The Jacobian $DF(x)$ inherits 
a block structure from the partition. The critical condition for proper 
lumping (Proposition~\ref{colslem}) requires that all column sums within 
each block $M_{pq}$ are equal: 
$\sum_{i \in I_p} M_{ij} = \sum_{i \in I_p} M_{i\ell}$ 
for all $j, \ell \in I_q$.}
\label{fig:proper_lumping}
\end{figure}

\begin{remark}\label{scalerem} {\em 
    From a practical perspective, a lumping approach with \eqref{properlumpmatgam} and arbitrary $\gamma_j>0$ will allow to search in a wider range, due to the additional scaling parameters. This is noted in Pepiot et al.\ \cite{PCP2019}, for instance.
    To illustrate this, we look at the effect of scaling for a quadratic differential system
       \begin{equation}\label{quadunsca}
        \dot{x_i}=\sum_j\lambda_{ij}\, x_j+\sum_{j,k} \alpha_{ijk}\, x_j x_k;\quad \alpha_{ijk}=\alpha_{ikj}.
    \end{equation}
    
    For the scaled variables 
    \[
    \widetilde x_i=\gamma_ix_i
    \]
    this system becomes
     \begin{equation}\label{quadsca}
        \dot{\widetilde x_i}=\sum_j\widetilde \lambda_{ij}\,\widetilde x_j+\sum_{j,k} \widetilde \alpha_{ijk}\,\widetilde x_j\widetilde x_k,
    \end{equation}
    
    with 
    \begin{equation}\label{tildepars}
     \widetilde \lambda_{ij}=\gamma_j^{-1}\gamma_i \lambda_{ij},\quad \widetilde \alpha_{ijk} =\gamma_i\gamma_j^{-1}\gamma_k^{-1}\alpha_{ijk}.
    \end{equation}
    Thus one may test whether the column sum condition in Proposition \ref{colslem} holds for \eqref{tildepars} with suitable $\gamma_j$.}
\end{remark}

\subsection{Lumpings induced by species permutations}
Classical approaches to lumping are based on arguments from chemistry. For instance, in the introductory paragraph of Wei and Kuo \cite{WK1969a}, the authors mention that grouping species into equivalence classes is a common practice. They specifically mention the PONA analysis (with paraffins, olefins, naphtenes and aromatics as classes) in petroleum processing. Notably, Wei and Kuo focus attention on kinetic lumpability conditions. With this background, we consider mathematical procedures that
lump ``species which behave alike with regard to kinetics'', informally speaking. A natural mathematical interpretation, which we adopt here, is to require that switching equivalent species will produce (mutatis mutandis) the same reaction equations.\\
Thus consider a permutation $\pi$ of species $X_1,\ldots,X_n$. This permutation may be represented by a matrix $P$ such that
\[
x=\begin{pmatrix} x_1\\ \vdots \\ x_n\end{pmatrix}\mapsto P\,x.
\]
\begin{definition}
    We say that this permutation respects complexes if for every complex $\sum_j y_{ji}X_j$ the linear combination $\sum_j y_{ji} \pi(X_j)$ is also a complex of the reaction network. 
\end{definition} 
This is a rather restrictive condition on the permutation. 
\begin{remark}
{\em    We note a different way to state this property: A permutation respects complexes if and only if there exists a permutation matrix $\widehat P$ such that 
\[
P\,Y=Y\,\widehat P.
\]
In this case one verifies 
\[
\left(P\,x\right)^Y= \widehat P\,x^Y.
\] 
}
\end{remark}
 Clearly the permutations which respect all complexes of a given reaction network form a group. In the following, let $P$ represent an element of this group.
\begin{lemma}
   The linear transformation $P$ maps solutions of the differential equation \eqref{eq:ODE2A} to solutions of  
   \begin{equation}\label{eq:ODE2B}
\dot{x}=  Y \widetilde{A(k)}\,  x^Y,\qquad \widetilde{A(k)}:=\widehat P\,A(k){\widehat P}^{-1}.
\end{equation}
 In particular $P$ is a symmetry of \eqref{eq:ODE2A} if and only if $Y\widehat P \,A(k)\,{\widehat P}^{-1}=YA(k)$. 
\end{lemma} 
\begin{proof} We have 
\[
\frac d{dt} Px=PYA(k)x^Y=Y\,\widehat P\,A(k)\,{\widehat P}^{-1}\,\widehat Px^Y=Y\,\widehat P\,A(k)\,{\widehat P}^{-1}\,(Px)^Y,
\]
and the assertion follows by this chain of equalities, combined with the criterion for solution-preserving maps.
\end{proof}
\begin{remark}{\em  
    A permutation that respects complexes gives rise to an automorphism of the (unlabeled) directed graph of the reaction network. Symmetry then imposes restrictions on the rate constants.
The symmetry condition is certainly satisfied when $P$ corresponds to an automorphism of the labeled graph, but this condition may not be necessary, depending on the graph. Such a symmetry-based approach has been applied successfully to epidemic models on networks, 
where graph automorphisms induce exact lumping (Simon et al.\ \cite{SimonTaylorKiss2011}).}
\end{remark}
We now describe a heuristic that establishes a correspondence to lumping: Given a finite linear symmetry group, there exists a natural reduction by nonlinear polynomial invariants; see Sturmfels \cite{Sturm}, but this will not lead to a reduction of dimension. But it is worth a try to single out the {\em linear} invariants and choose these to construct a linear lumping map, with parameter conditions then determined via Lemma \ref{lalem}. Thus one arrives at proper lumpings and critical parameter conditions.
\begin{example}{\em 
    This example serves solely for the purpose of illustration. For the simple reaction network
    \begin{equation*}
   \ce{$X_1$ + $X_2$ <=>[$k_1$][$k_{-1}$] $X_3$}
\end{equation*}
with corresponding differential equation
\[
\begin{array}{rcl}
    \dot x_1&=&-k_1x_1x_2+k_{-1}x_3\\
    \dot x_2&=&-k_1x_1x_2+k_{-1}x_3\\
    \dot x_3&=&k_1x_1x_2-k_{-1}x_3\\
\end{array}
\]
the switching of $X_1$ and $X_2$ yields a complex-respecting permutation. The algebra of polynomials that are invariant with respect to this group action is generated by 
\[
z_1=x_1+x_2,\quad z_2=x_3,\quad z_3=x_1x_2,
\]
and the symmetry reduction yields the system
\[
\begin{array}{rcl}
    \dot z_1 &=&-2k_1z_3+2k_{-1}z_2  \\
     \dot z_2 &=&k_1z_3-k_{-1}z_2  \\
      \dot z_3 &=&-k_1z_1z_3+k_{-1}z_1z_2  \\
\end{array}
\]
with no reduction of dimension.\\
The heuristic approach is to build a candidate for a lumping map from the degree one invariants $x_1+x_2$ and $x_3$, thus
\[
T=\begin{pmatrix}
    1&1&0\\ 0&0&1
\end{pmatrix}, \quad B=\begin{pmatrix}
    1\\ -1\\ 0
\end{pmatrix}.
\]
The condition from Lemma \ref{lalem} is computed as
\[
k_1(x_1-x_2)\begin{pmatrix}
    2\\-1
\end{pmatrix}=0; \text{thus  } k_{1}=0.
\]
We get a linear lumping when the forward reaction is absent, but only then.}
\end{example}

\begin{example}{\em 
For a more substantial example, we again look at the reversible Michaelis--Menten network. The permutation that switches substrate and product while fixing enzyme and complex induces an automorphism of the unlabeled  graph. We go through the formalities:\\
We have species $X_1=S,\,X_2=E,\,X_3=C,\,X_4=P$ and complexes $Y_1=S+E$, $Y_2=C$, $Y_3=E+P$, and therefore
\[
Y=\begin{pmatrix} 1&0&0\\ 1&0&1\\0&1&0\\ 0&0&1\end{pmatrix};\text{  moreover  } A(k)=\begin{pmatrix} -k_1 &k_{-1}&0\\
k_1&-(k_{-1}+k_2) & k_{-2}\\ 0& k_{2}&-k_{-2}\end{pmatrix}.
\]
The nontrivial permutation corresponds to the matrix
\[
R=\begin{pmatrix} 0&0&0&1\\ 0&1&0&0\\ 0&0&1&0\\ 1&0&0&0\end{pmatrix},
\]
and we find 
\[
\widehat R=\begin{pmatrix} 0&0&1\\ 0&1&0\\1&0&0\end{pmatrix}.
\]
Furthermore

\[
\widetilde{A(k)}=\begin{pmatrix} -k_{-2} &k_{2}&0\\
k_{-2}&-(k_{-1}+k_2) & k_1\\ 0& k_{-1}&-k_{1}\end{pmatrix},
\]

hence the symmetry condition amounts to $k_1=k_{-2}$ and $k_{-1}=k_2$. The symmetry of the differential equation thus corresponds to a symmetry of the labeled graph.\\
The invariant algebra of the permutation group is generated by $z_1=x_1+x_4$, $z_2=x_2$, $z_3=x_3$ and $z_4=x_1x_4$. The linear generators already provide a solution-preserving map to the three dimensional system
\[
\begin{array}{rcl}
   \dot z_1  &=& -k_1z_1z_2+(k_{-1}+k_2)z_3  \\
      \dot z_2 &=& -k_1z_1z_2+(k_{-1}+k_2)z_3  \\
      \dot z_3 &=& k_1z_1z_2-(k_{-1}+k_2)z_3  \\
\end{array}.
\]
The parameter condition is included in (but more restrictive than) the one in Example \ref{ex:mm}; with essentially the same reduced system.
}
\end{example}

\begin{remark}[Connection to bisimulation]\label{bisimrem}{\em
The proper lumping conditions have a natural interpretation in terms of \emph{bisimulation} from theoretical computer science. Two species are bisimilar if they have identical ``behavior'' in a precise sense. For differential equations, \emph{backward differential equivalence} (BDE) requires that equivalent species have identical dynamics from identical initial conditions; \emph{forward differential equivalence} (FDE) requires that sums of equivalent species have identical sum dynamics, see Cardelli et al. \cite{CardelliTTV2017}. \\
Proper lumping corresponds to FDE: the condition that all column sums in each Jacobian block are equal ensures that the sum $y_p = \sum_{i \in I_p} x_i$ satisfies a closed differential equation. The ERODE algorithm (Cardelli et al.\ \cite{CardelliTTV2019}) efficiently computes the coarsest partition satisfying FDE using partition refinement techniques adapted from Markov chain minimization.\\
The BDE condition means invariance of the set where, for each group of equivalent species, their concentrations are equal.}
\end{remark}

\section{Case studies}

In this section we present case studies of two biochemically relevant systems. Our primary purpose is to illustrate the computation of critical parameters, and subsequent reduction, with a focus on algorithmic considerations. We will sketch a few illustrations to ensure the viability and relevance of the method. But a more detailed study, including computations and biological interpretation will be taken up in a future paper. Some of the reduced systems may look quite underwhelming at first sight, but the main interest should lie in small perturbations of critical parameters, as we show by one example. These matters will also be dealt with in a future paper.

\subsection{A self-replication model}  \label{sec:replication}
We apply the critical parameter framework to a biochemical self-replication 
model that motivated the present investigation.
We consider the replication mechanism studied by Gijima and Peacock-L\'opez~\cite{GPL2020}, 
which extends earlier work on minimal self-replicating 
systems (see Peacock-L\'opez \cite{PeacockLopez2001}, Beutel and Peacock-L\'opez \cite{BeutelPL2006}). The mechanism involves two ``food'' 
species $A$ and $B$, a product (template) $P$, and three intermediates $I_a$, $I_b$, 
and $I$. We denote these as $X_1 = A$, $X_2 = B$, $X_3 = P$, 
$X_4 = I_a$, $X_5 = I_b$, and $X_6 = I$.

The reaction network consists of five reversible reactions (the last one is assumed irreversible in \cite{GPL2020}):
\begin{align}
\ce{$X_1$ + $X_3$ & <=>[$k_1$][$k_{-1}$] $X_4$} \label{reac1}\\
\ce{$X_2$ + $X_4$ & <=>[$k_2$][$k_{-2}$] $X_6$} \label{reac2}\\
\ce{$X_2$ + $X_3$ & <=>[$k_3$][$k_{-3}$] $X_5$} \label{reac3}\\
\ce{$X_1$ + $X_5$ & <=>[$k_4$][$k_{-4}$] $X_6$} \label{reac4}\\
\ce{$X_6$ & <=>[$k_5$][$k_{-5}$] $2X_3$.} \label{reac5}
\end{align}

The key feature is \emph{autocatalysis}: the product $P = X_3$ appears on the 
left-hand side of reactions \eqref{reac1} and \eqref{reac3} (as a reactant that 
facilitates intermediate formation) and is regenerated with gain in reaction~\eqref{reac5}. 
The net effect is that $P$ catalyzes its own production, giving rise to the 
self-replication dynamics characteristic of these systems.

With mass action kinetics, the time evolution of this network is governed by the differential equation system
\begin{equation}\label{replicationsys}
\begin{array}{rcl}
\dot x_1 & =& -k_1x_1x_3+ k_{-1} x_4 -k_4x_1x_5 +k_{-4}x_6\\[2pt]
\dot x_2 &= & -k_2x_2x_4 +k_{-2}x_6-k_3x_2x_3+k_{-3} x_5\\[2pt]
\dot x_3&=&  -k_1x_1x_3+ k_{-1} x_4 -k_3x_2x_3+k_{-3} x_5 + 2k_5x_6-2k_{-5}x_3^2 \\[2pt]
\dot x_4&=& k_1x_1x_3- k_{-1} x_4  -k_2x_2x_4 +k_{-2}x_6\\[2pt]
\dot x_5&=&-k_4x_1x_5 +k_{-4}x_6+k_3x_2x_3-k_{-3} x_5   \\[2pt]
\dot x_6&=&k_2x_2x_4 -k_{-2}x_6+k_4x_1x_5 -k_{-4}x_6- k_5x_6+k_{-5}x_3^2.
\end{array}
\end{equation}
This six-dimensional system has two stoichiometric first integrals:
\begin{equation}\label{stoichintegrals}
\mu_1 = x_1 - x_2 + x_4 - x_5, \qquad \mu_2 = x_2 + x_3 + x_4 + 2x_5 + 2x_6.
\end{equation}
The first integral $\mu_1$ reflects the balance between the two ``pathways'' through 
intermediates $I_a$ and $I_b$, while $\mu_2$ represents conservation of total 
``building blocks.''

\paragraph{Setup.}
We seek a linear lumping that preserves the stoichiometric first integrals and reduces 
the system to dimension three; hence effectively to one-dimensional dynamics. Applying row reduction (Remark~\ref{rrerem}), and following Remark~\ref{remconstr},
we consider a lumping matrix of the form
\begin{equation}\label{replumpT}
T = \begin{pmatrix}
0 & 0 & 1 & p & q & r \\
1 & -1 & 0 & 1 & -1 & 0 \\
0 & 1 & 1 & 1 & 2 & 2
\end{pmatrix},
\end{equation}
where the second and third rows encode the stoichiometric first 
integrals~\eqref{stoichintegrals}, and the first row introduces a new observable
\[
y_1 = x_3 + p\, x_4 + q\, x_5 + r\, x_6
\]
that combines the product concentration with weighted intermediate concentrations. 
The parameters $p$, $q$, and $r$ are to be determined.

Following the procedure of Section 4, we apply Lemma \ref{lalem} to determine critical 
parameters. With the lumping matrix \eqref{replumpT}, we compute a basis for $\ker T$:
\[
B = \begin{pmatrix}
p-2 & q-1 & r-2 \\
p-1 & q-2 & r-2 \\
-p & -q & -r \\
1 & 0 & 0 \\
0 & 1 & 0 \\
0 & 0 & 1
\end{pmatrix}.
\]
The lumping condition $T \cdot DF(x,k) \cdot B = 0$ for all $x$ yields a system of 
polynomial equations in the rate parameters and the lumping parameters $p$, $q$, $r$.

According to Proposition \ref{algocrit}, the conditions can be represented as a homogeneous linear system for the vector of rate parameters, with matrix entries quadratic in the entries of $T$. Ordering the rate constants as 
\[
\begin{pmatrix}
    k_1&k_2&k_3&k_4&k_5&k_{-1}&k_{-2}&k_{-3}&k_{-4}&k_{-5}
\end{pmatrix}
\]
and using {\sc Mathematica}\texttrademark, one finds the matrix
{\tiny
\[
\begin{pmatrix}
-((p-1) r) & 0 & 0 & 0 & 0 & 0 & 0 & 0 & 0 & 0 \\
0 & 0 & -((q-1) r) & 0 & 0 & 0 & 0 & 0 & 0 & 0 \\
p (r-2)-r+2 & 0 & q (r-2)-r+2 & 0 & 0 & 0 & 0 & 0 & 0 & -2 (r-2) r \\
0 & -((r-2) (p-r)) & 0 & 0 & 0 & 0 & 0 & 0 & 0 & 0 \\
0 & 0 & 0 & -((r-2) (q-r)) & 0 & 0 & 0 & 0 & 0 & 0 \\
0 & 0 & 0 & 0 & 2-r & 0 & p-r & 0 & q-r & 0 \\
q-p q & 0 & 0 & r-q & 0 & 0 & 0 & 0 & 0 & 0 \\
0 & 0 & -((q-1) q) & 0 & 0 & 0 & 0 & 0 & 0 & 0 \\
p q-p-q+1 & 0 & q^2-3 q+2 & 0 & 0 & 0 & 0 & 0 & 0 & 4 q-2 q r \\
0 & -((q-2) (p-r)) & 0 & 0 & 0 & 0 & 0 & 0 & 0 & 0 \\
0 & 0 & 0 & -((q-1) (q-r)) & 0 & 0 & 0 & 0 & 0 & 0 \\
0 & 0 & 0 & 0 & 0 & 0 & 0 & 1-q & 0 & 0 \\
-((p-1) p) & 0 & 0 & 0 & 0 & 0 & 0 & 0 & 0 & 0 \\
0 & r-p & p-p q & 0 & 0 & 0 & 0 & 0 & 0 & 0 \\
p^2-3 p+2 & 0 & p q-p-q+1 & 0 & 0 & 0 & 0 & 0 & 0 & 4 p-2 p r \\
0 & -((p-1) (p-r)) & 0 & 0 & 0 & 0 & 0 & 0 & 0 & 0 \\
0 & 0 & 0 & -((p-2) (q-r)) & 0 & 0 & 0 & 0 & 0 & 0 \\
0 & 0 & 0 & 0 & 0 & 1-p & 0 & 0 & 0 & 0 \\
0 & 0 & 0 & 0 & 0 & 0 & 0 & 0 & 0 & 0 \\
0 & 0 & 0 & 0 & 0 & 0 & 0 & 0 & 0 & 0 \\
0 & 0 & 0 & 0 & 0 & 0 & 0 & 0 & 0 & 0 \\
0 & 0 & 0 & 0 & 0 & 0 & 0 & 0 & 0 & 0 \\
0 & 0 & 0 & 0 & 0 & 0 & 0 & 0 & 0 & 0 \\
0 & 0 & 0 & 0 & 0 & 0 & 0 & 0 & 0 & 0 \\ 
\end{pmatrix}
\]
}

A further analysis with {\sc Singular} yields a complete set of solutions here, with 22 cases; see appendix, subsection \ref{sec:61comps}. One observation is that some rate parameters must be zero in every case, and in many cases one effectively obtains a generic setting for a smaller reaction network.

Computing the reduced system (according to Remark \ref{redcomprem}) will yield a three dimensional system for variables $y_1,\,y_2,\, y_3$; the built-in constraints will always yield $\dot y_2=\dot y_3=0$.

Here the sparse form of the matrix also permits (in part) an approach ``by hand''. For instance, all but the first entries of the row 1 of the matrix are automatically zero, and this shows that $k_1=0$ or $(p-1)r=0$. Likewise, row 12 shows that $k_{-3}=0$ or $q=1$.\\
We single out some interesting scenarios for which relatively few rate parameters must vanish: Proceeding with $p=q=1$
(which implies equal weighting of the two intermediate pathways), 
leaves the following cases:
\begin{itemize}
    \item For $r\neq 1$ and $r\neq 2$ one finds $k_{2}=k_{4}=k_{-5}=0$ and
    \[
    (2-r)k_5+(1-r)k_{-2}+(1-r)k_{-4}=0.
    \]
    Here one has either $k_5=k_{-2}=k_{-4}=0$, or positivity of some rate constants forces $1<r<2.$
    \item In case $r=1$ one finds $k_5=k_{-5}=0$.
    \item In case $r=2$ one finds $k_{2}=k_{4}=k_{-2}+k_{-4}=0$; with non-negativity the latter condition means $k_{-2}=k_{-4}=0$.
\end{itemize}
In all these cases the reduced system will also yield $\dot y_1=0$; hence one may say that at the critical parameters one gains a ``hidden conservation law'' $y_1={\rm const.}$, yielding 
\[
x_3 + p\, x_4 + q\, x_5 + r\, x_6={\rm const.}+O(\varepsilon)
\]
for $O(\varepsilon)$ perturbations of the critical parameters, on fixed finite time intervals.
\begin{example}\label{ex:pqr}{\em
We look at
the case $p=q=r=1$ in more detail, proceeding according to item 3 of Remark \ref{redcomprem}: Augmenting the linear forms $y_1,\,y_2,\,y_3$ from the lumping matrix by $y_4=x_4,\,y_5=x_5$ and $y_6=x_6$, we have a linear coordinate change
\[
\begin{array}{rcl}
x_1 &=&-y_1+y_2+y_3-y_4-y_6\\
                    x_2&=&-y_1 + y_3 - y_5 - y_6\\
                    x_3 &=& y_1 - y_4 - y_5 - y_6\\
  x_4 &=& y_4\\
                            x_5&=& y_5\\
                            x_6 &=&y_6
\end{array}     
\]
With perturbed parameters $k_5=\varepsilon k_5^*$ and $k_{-5}=\varepsilon k_{-5}^*$, the system in lumping-adapted coordinates will be
\[
\begin{array}{rcl}
            \dot y_1&=& -\varepsilon k_{-5}^* (y_1 - y_4 - y_5 - y_6)^2  + \varepsilon k_5^* y_6\\
             \dot y_2&=& 0\\
               \dot y_3&=&0\\           
        \dot y_4&=& k_1 (-y_1 + y_2 + y_3 - y_4 - y_6) (y_1 - y_4 - y_5 - y_6)\\
      &&  - k_2 (-y_1 + y_3 - y_5 - y_6) y_4 - k_{-1} y_4 + k_{-2} y_6\\
     \dot y_5&=& k_3 (-y_1 + y_3 - y_5 - y_6) (y_1 - y_4 - y_5 - y_6)\\
       && - k_4 (-y_1 + y_2 + y_3 - y_4 - y_6) y_5 - k_{-3} y_5 + k_{-4} y_6\\
      \dot y_6&=&k_2 (-y_1 + y_3 - y_5 - y_6) y_4
   + k_4 (-y_1 + y_2 + y_3 - y_4 - y_6) y_5\\&&- k_{-2} y_6 - k_{-4} y_6
        + \varepsilon k_{-5}^* (y_1 - y_4 - y_5 - y_6)^2 
   - \varepsilon k_5^* y_6 
   \end{array}
\]
At $\varepsilon=0$, the first three equations reflect the reduced system. Assigning some constant values to the first integrals $y_2$ and $y_3$, one is left with a four dimensional system with a small parameter $\varepsilon$. \\
In fact, we have uncovered singular perturbation scenarios (see Tikhonov \cite{tikh}, Fenichel \cite{Fenichel1979}): Fixing constant values $c_1$ for $y_2+y_3$ and $c_2$ for $y_3$, we obtain
\[
\begin{array}{rcl}
            \dot y_1&=& -\varepsilon k_{-5}^* (y_1 - y_4 - y_5 - y_6)^2  + \varepsilon k_5^* y_6\\      
        \dot y_4&=& k_1 (-y_1 + c_1 - y_4 - y_6) (y_1 - y_4 - y_5 - y_6)
        - k_2 (-y_1 + c_2 - y_5 - y_6) y_4 \\
        &&- k_{-1} y_4 + k_{-2} y_6\\
     \dot y_5&=& k_3 (-y_1 + c_2 - y_5 - y_6) (y_1 - y_4 - y_5 - y_6)
        - k_4 (-y_1 + c_1 - y_4 - y_6) y_5\\
        &&- k_{-3} y_5 + k_{-4} y_6\\
      \dot y_6&=&k_2 (-y_1 + c_2 - y_5 - y_6) y_4
   + k_4 (-y_1 + c_1 - y_4 - y_6) y_5 - k_{-2} y_6 - k_{-4} y_6\\
        &&+ \varepsilon k_{-5}^* (y_1 - y_4 - y_5 - y_6)^2 
   - \varepsilon k_5^* y_6 
   \end{array}
\]
For instance, this system admits the stationary point $y_1=y_4=y_5=y_6=0$, which has non-invertible Jacobian when $\varepsilon=0$. The Jacobian of the last three equations with respect to $y_4,\,y_5,\,y_6$ at the stationary point $0$ turns out to be 

\[
\begin{pmatrix}
    -k_1c_1-k_2c_2-k_{-1}& -k_1c_1& -k_1c_1+k_{-2}\\
    -k_3c_2& -k_3c_2-k_4c_1-k_{-3}& -k_3c_2+k_{-4}\\
    k_2c_2&k_4c_1& -k_{-2}-k_{-4}
\end{pmatrix}.
\]

Choosing all $k_i>0$ and noting that $c_1\geq 0,\,c_2\geq 0$ in physically relevant cases, one sees that all the eigenvalues of this matrix have negative real part in some open region of parameter space\footnote{Presumably this condition holds for all positive parameters, but the verification is not trivial.}. By the implicit function theorem, when $\varepsilon=0$  there exists a local parameterization  of $y_4, y_5$ and $y_6$ as functions of $y_1$. Moreover, for small positive $\varepsilon$, Tikhonov-Fenichel theory guarantees the existence of a reduced equation for $y_1$, with the leading order term obtained from the first equation by substituting the expressions of $y_4,\,y_5,\,y_6$ as functions of $y_1$. (We mention also the implicit version in Goeke and Walcher \cite{gw2}.) We will not carry out any details here (this will be done in a future paper), but we are satisfied to illustrate another use of linear lumping.}
\end{example}
\begin{example}{\em 
Let us also consider a case for which the reduced system has nonzero right-hand side. We choose component [9] from the Appendix, subsection \ref{sec:61comps}, with
\[
k_1=k_3=k_5=k_{-1}=k_{-3}=0\quad\text{and}\quad p=q=r=0.
\]
The remaining system is
\begin{equation*}\label{replicationsyscut}
\begin{array}{rcl}
\dot x_1 & =&  -k_4x_1x_5 +k_{-4}x_6\\[2pt]
\dot x_2 &= & -k_2x_2x_4 +k_{-2}x_6\\[2pt]
\dot x_3&=&  -2k_{-5}x_3^2 \\[2pt]
\dot x_4&=&  -k_2x_2x_4 +k_{-2}x_6\\[2pt]
\dot x_5&=&-k_4x_1x_5 +k_{-4}x_6  \\[2pt]
\dot x_6&=&k_2x_2x_4 -k_{-2}x_6+k_4x_1x_5 -k_{-4}x_6+k_{-5}x_3^2.
\end{array}
\end{equation*}
With $y_1=x_3$ and the first integrals $y_2=\mu_1$, $y_3=\mu_2$ one obtains the reduced system
\begin{equation*}
    \begin{array}{rcl}
       \dot y_1  & =& -2k_{-5} y_1^2 \\
        \dot y_2 & =&0\\
        \dot y_3&=& 0
    \end{array}
\end{equation*}
}
\end{example}

\subsection{A two-pathway enzyme mechanism}\label{subsec:enzyme}
Enzymes frequently exhibit multiple pathways for substrate binding and catalysis. 
We consider a minimal model capturing this two-pathway structure (an extension of 
the Michaelis--Menten system), which seems analogous to the self-replication model 
by Gijima and Peacock-L\'opez~\cite{GPL2020}, while having a distinct biochemical
interpretation.


Consider an enzyme $E$ that converts substrate $S$ to product $P$ via two 
parallel pathways. Each pathway involves an initial enzyme-substrate 
complex ($C_1$ or $C_2$) that can transition to a common catalytically 
active complex $C$, which then releases product and regenerates the enzyme.

The reaction network consists of seven reactions, with only the last one irreversible:
\begin{align}
\ce{S + E &<=>[$k_1$][$k_{-1}$] C_1} \label{enz:reac1}\\
\ce{C_1 &->[$k_2$] C} \label{enz:reac2}\\
\ce{C &->[$k_{-2}$] C_1} \label{enz:reac2rev}\\
\ce{S + E &<=>[$k_3$][$k_{-3}$] C_2} \label{enz:reac3}\\
\ce{C_2 &->[$k_4$] C} \label{enz:reac4}\\
\ce{C &->[$k_{-4}$] C_2} \label{enz:reac4rev}\\
\ce{C &->[$k_5$] E + P.} \label{enz:reac5}
\end{align}
The two pathways ($S + E \rightleftharpoons C_1 \to C$ and 
$S + E \rightleftharpoons C_2 \to C$) converge on the common 
catalytic complex $C$, which can either proceed to product formation 
or revert to one of the initial complexes.
We will consider the system with the last reaction also reversible.

Renaming the species as $X_1 = S$, $X_2 = E$, $X_3 = C_1$, 
$X_4 = C_2$, $X_5 = C$, and $X_6 = P$, with mass action kinetics, the system is governed by

\begin{equation}\label{enzymeODErev}
\begin{array}{rcl}
\dot x_1 &=& -k_1 x_1x_2 + k_{-1} x_3 - k_3 x_1x_2 + k_{-3} x_4 \\[3pt]
\dot x_2 &=& -k_1 x_1x_2 + k_{-1} x_3 - k_3 x_1x_2 + k_{-3} x_4 + k_5 x_5 -k_{-5}x_2x_6\\[3pt]
\dot x_3 &=& k_1 x_1x_2 - k_{-1} x_3 - k_2 x_3 + k_{-2} x_5 \\[3pt]
\dot x_4 &=& k_3 x_1x_2 - k_{-3} x_4 - k_4 x_4 + k_{-4} x_5 \\[3pt]
\dot x_5 &=& k_2 x_3 + k_4 x_4 - k_{-2} x_5 - k_{-4} x_5 - k_5 x_5 {+k_{-5}x_2x_6}\\[3pt]
\dot x_6 &=& k_5 x_5 -k_{-5}x_2x_6.
\end{array}
\end{equation}

The system admits two stoichiometric first integrals:
\begin{equation}\label{enz:stoich}
\mu_1 = x_2+x_3+x_4+x_5, \qquad \mu_2 = x_1+x_3+x_4+x_5+x_6.
\end{equation}
The integral $\mu_1$ represents \emph{total enzyme conservation}: 
the enzyme cycles through various bound states but is never created or 
destroyed. The integral $\mu_2$ represents \emph{substrate-product 
balance}: substrate is converted to product while passing through 
enzyme-bound intermediates.

We seek a linear lumping reduction to dimension three via critical parameters.
We proceed according to Remark \ref{rrerem}, stipulating the first, third and fourth column to be linearly independent, thus ensuring that substrate and both intermediate complexes will appear in the reduced system.

Thus we get the matrices
\[
T=\begin{pmatrix}
    1&t_1&0&0&t_2&t_3\\
    0&t_4&1&0&t_5&t_6\\
    0&t_7&0&1&t_8&t_9
\end{pmatrix},
\]
and 
\[
B=\begin{pmatrix}
    t_1&t_2&t_3\\
    -1&0&0\\
    t_4&t_5&t_6\\
    t_7&t_8&t_9\\
    0&-1&0\\
    0&0&-1
\end{pmatrix}
\]
The critical parameter condition yields an ideal with generators 
\[
\begin{array}{ll}
    & \langle (k_1+k_3)\cdot(1+t_1), \\
    & (k_1+k_3)\cdot t_1\cdot(1+t_1),\\
    & {k_{-5}\cdot(t_1-t_2+t_3)},\\
    &k_{-1}\cdot t_4+k_{-1}\cdot t_1\cdot t_4+k_2\cdot t_2\cdot t_4+ k_{-3}\cdot t_7+k_{-3}\cdot t_1\cdot t_7+k_4\cdot t_2\cdot t_7,\\
   & (k_1+k_3)\cdot (1+t_1)\cdot t_2,\\
    & -k_5\cdot t_1+k_5\cdot t_2+k_{-2}\cdot t_2+k_{-4}\cdot t_2-k_5\cdot t_3+k_{-1}\cdot t_5+k_{-1}\cdot t_1\cdot t_5+k_2\cdot t_2\cdot t_5\\+&k_{-3}\cdot t_8+k_{-3}\cdot t_1\cdot t_8+k_4\cdot t_2\cdot t_8,\\
    &k_{-5}\cdot t_1{-k_{-5}\cdot t_2}-k_1\cdot t_3-k_3\cdot t_3+k_{-5}\cdot t_3-k_1\cdot t_1\cdot t_3-k_3\cdot t_1\cdot t_3,\\
    &k_{-1}\cdot t_6+k_{-1}\cdot t_1\cdot t_6+k_2\cdot t_2\cdot t_6+k_{-3}\cdot t_9+k_{-3}\cdot t_1\cdot t_9+k_4\cdot t_2\cdot t_9,\\
    & -k_1+k_1\cdot t_4+k_3\cdot t_4,\\
    & t_1\cdot (-k_1+k_1\cdot t_4+k_3\cdot t_4),\\
    &{k_{-5}\cdot (t_4-t_5+t_6)},\\
    & -k_2\cdot t_4-k_{-1}\cdot t_4+k_{-1}\cdot t_4^2+k_2\cdot t_4\cdot t_5+k_{-3}\cdot t_4\cdot t_7+k_4\cdot t_5\cdot t_7,\\
    &t_2\cdot (-k_1+k_1\cdot t_4+k_3\cdot t_4),\\
    & -k_{-2}-k_5\cdot t_4-k_2\cdot t_5+k_5\cdot t_5-k_{-1}\cdot t_5+k_{-2}\cdot t_5+k_{-4}\cdot t_5+k_{-1}\cdot t_4\cdot t_5\\+&k_2\cdot t_5^2-k_5\cdot t_6+k_{-3}\cdot t_4\cdot t_8+k_4\cdot t_5\cdot t_8,\\
    & k_1\cdot t_3+k_{-5}\cdot t_4-k_1\cdot t_3\cdot t_4-k_3\cdot t_3\cdot t_4{-k_{-5}\cdot t_5}+k_{-5}\cdot t_6,\\
    &-k_2\cdot t_6-k_{-1}\cdot t_6+k_{-1}\cdot t_4\cdot t_6+k_2\cdot t_5\cdot t_6+k_{-3}\cdot t_4\cdot t_9+k_4\cdot t_5\cdot t_9,\\
    &-k_3+k_1\cdot t_7+k_3\cdot t_7,\\
    &-t_1\cdot (-k_3+k_1\cdot t_7+k_3\cdot t_7),\\
    &{k_{-5}\cdot (t_7-t_8+t_9)},\\
    &-k_4\cdot t_7-k_{-3}\cdot t_7+k_{-1}\cdot t_4\cdot t_7+k_{-3}\cdot t_7^2+k_2\cdot t_4\cdot t_8+k_4\cdot t_7\cdot t_8,\\
    &t_2\cdot (-k_3+k_1\cdot t_7+k_3\cdot t_7),\\
    &-k_{-4}-k_5\cdot t_7{+}k_{-1}\cdot t_5\cdot t_7-k_4\cdot t_8+k_5\cdot t_8+k_{-2}\cdot t_8-k_{-3}\cdot t_8+k_{-4}\cdot t_8\\+&k_2\cdot t_5\cdot t_8+{k_{-3}}\cdot t_7\cdot t_8+k_4\cdot t_8^2-k_5\cdot t_9,\\
    &k_3\cdot t_3+k_{-5}\cdot t_7-k_1\cdot t_3\cdot t_7-k_3\cdot t_3\cdot t_7{-k_{-5}\cdot t_8}+k_{-5}\cdot t_9,\\
    &k_{-1}\cdot t_6\cdot t_7+k_2\cdot t_6\cdot t_8-k_4\cdot t_9-k_{-3}\cdot t_9+k_{-3}\cdot t_7\cdot t_9+k_4\cdot t_8\cdot t_9\rangle.
\end{array}
\]
Computing with the routine {\tt facstd}  of  {\sc Singular} \cite{sing} a decomposition  of   this ideal can be obtained. The output yields 92  components. The first five are listed in subsection \ref{sec:62comp}, some of the others are of rather large size. Dealing with these is decidedly unwieldy.  But using {\sc Singular} again to eliminate the $t's$ and then computing the minimal associate primes of the elimination ideal with \texttt{minAssChar} one obtains a list of rather manageable conditions for critical parameter candidates, as follows. (The conditions in the {\sc Singular} output specify parameter combinations that must be zero; for instance in component [4] the requirement is $k_{-2}=k_1=0$.)

\begin{verbatim}
[1]:
   _[1]=km5
   _[2]=k1^2*k2^2*k3*k4+k1*k2^2*k3^2*k4- k1^2*k2*k3*k4^2-
   k1*k2*k3^2*k4^2- k1^2*k2^2*k3*k5+k1^2*k2*k3*k4*k5-
   k1*k2*k3^2*k4*k5+  k1*k3^2*k4^2*k5+k1^2*k2*k3*k4*km1+
   2*k1*k2*k3^2*k4*km1- k1*k3^2*k4^2*km1-k1^2*k2*k3*k5*km1-
   k1*k3^2*k4*k5*km1+k1*k3^2*k4*km1^2+k1^2*k2*k3*k4*km2+
   k1*k2*k3^2*k4*km2+k1*k3^2*k4^2*km2+k3^3*k4^2*km2-
   k1^2*k2*k3*km1*km2-k1*k3^2*k4*km1*km2+k1^2*k2^2*k3*km3-
   2*k1^2*k2*k3*k4*km3-k1*k2*k3^2*k4*km3+k1^2*k2*k3*k5*km3+
   k1*k3^2*k4*k5*km3+k1^2*k2*k3*km1*km3-k1*k3^2*k4*km1*km3+
   k1^2*k2*k3*km2*km3+k1*k3^2*k4*km2*km3- k1^2*k2*k3*km3^2-
   k1^3*k2^2*km4-k1^2*k2^2*k3*km4-k1^2*k2*k3*k4*km4-
   k1*k2*k3^2*k4*km4-k1^2*k2*k3*km1*km4-
   k1*k3^2*k4*km1*km4+k1^2*k2*k3*km3*km4+k1*k3^2*k4*km3*km4
[2]:
   _[1]=km5
   _[2]=k1*k2*k4+k2*k3*k4+k3*k4*km1+k1*k2*km3
[3]:
   _[1]=km4
   _[2]=k3
[4]:
   _[1]=km2
   _[2]=k1
[5]:
   _[1]=k4*km1*km2+km1*km2*km3+k2*km3*km4+km1*km3*km4
   _[2]=k3*km1*km2+k1*k2*km4+k2*k3*km4+k3*km1*km4
   _[3]=k1*k4*km2+k3*k4*km2+k1*km2*km3+k1*km3*km4
   _[4]=k1*k2*k4+k2*k3*k4+k3*k4*km1+k1*k2*km3
[6]:
   _[1]=k2-k4+km1-km3
   _[2]=k3*km2-k1*km4
\end{verbatim}

We discuss these components briefly. 
\begin{enumerate}
    \item 
The lumping matrices for components [2], [3], [4] and [5] actually yield generic networks comprising fewer reactions. This is obvious for component [4] (where the reactions with product $C_1$ are ``switched off'') and likewise for component [3]. For component [2], the second condition, in view of non-negativity of the rate constants, yields 
\[
k_1k_2k_4=k_2k_3k_4=k_3k_4k_{-1}=k_1k_2k_{-3}=0,
\]
and again this leads to generic settings for smaller networks. A similar argument applies to component [5]. \\
To further analyze these systems, the algorithm from subsection \ref{sec:genconstr} is readily available to identify Type 1 subspaces and subsequent reductions.\\
For instance, we look at the case $k_3=k_{-4}=0$,
with differential equation system
\begin{equation*}
\begin{array}{rcl}
\dot x_1 &=& -k_1x_1x_2+  k_{-1} x_3 + k_{-3} x_4 \\[3pt]
\dot x_2 &=& -k_1x_1x_2+k_{-1} x_3  + k_{-3} x_4 + k_5 x_5 -k_{-5}x_2x_6\\[3pt]
\dot x_3 &=&  k_1x_1x_2- k_{-1} x_3 - k_2 x_3 + k_{-2} x_5 \\[3pt]
\dot x_4 &=&  - k_{-3} x_4 -k_4x_4 \\[3pt]
\dot x_5 &=& k_2 x_3 +k_4x_4 - k_{-2} x_5 - k_5 x_5 +k_{-5}x_2x_6\\[3pt]
\dot x_6 &=& k_5 x_5 -k_{-5}x_2x_6.
\end{array}
\end{equation*}
From a chemical perspective, the switched-off reactions mean that the network essentially describes a single-pathway enzyme reaction with two intermediate complexes $C_1$ and $C$.
Applying the algorithm for Type 1 subspaces (starting with every species), one finds that only $Tx=x_4$ will provide a reduction (which is obvious from the shape of the differential equation). The case $k_1=k_{-2}=0$ is analogous, with $C_2$ and $C$.
\item Considering component [1], this looks rather challenging. But we note that choosing all $k_i$ ($i\not=-5$) equal and positive will satisfy the second condition, so for this scenario there exist linear lumping reductions that are physically meaningful.
\item Component [6] has a direct interpretation.  Its two equations can be written as $k_{-1}+k_2=k_{-3}+k_4$ and $k_3k_{-2}=k_1k_{-4}$.  The first equality balances the total exit rates from the two branch complexes $C_1$ and $C_2$, while the second equality says that the pathway bias is the same for association from $S+E$ and for return from the common complex $C$.  Substituting these relations into the Li--Rabitz equations in the row-echelon chart e.g. in case $k_1>0, \,k_3>0$ yields the lumping matrix
\[
T=
\begin{pmatrix}
1&-1&0&0&0&1\\
0&\dfrac{k_1}{k_1+k_3}&1&0&\dfrac{k_1}{k_1+k_3}&0\\
0&\dfrac{k_3}{k_1+k_3}&0&1&\dfrac{k_3}{k_1+k_3}&0
\end{pmatrix}.
\]
The corresponding variables are
\[
y_1=x_1-x_2+x_6,
\quad
 y_2=x_3+\frac{k_1}{k_1+k_3}(x_2+x_5),
\quad
 y_3=x_4+\frac{k_3}{k_1+k_3}(x_2+x_5),
\]
and the reduced system is
\[
\dot y_1=0,
\quad
\dot y_2=\frac{k_{-1}+k_2}{k_1+k_3}(k_1y_3-k_3y_2),
\quad
\dot y_3=-\frac{k_{-1}+k_2}{k_1+k_3}(k_1y_3-k_3y_2).
\]
Every solution of this system converges to a point $(c_1, c_2k_1,c_2k_3)^{\rm tr}$, from which the asymptotic behavior of $x_3$ and $x_4$ can be obtained.
\end{enumerate}

\section{Discussion and outlook}
We developed a systematic theory of linear lumping for parameter-dependent mass 
action networks, centered on the distinction between generic and non-generic parameter regimes. \\
On the one hand, for generic parameters---those ranging in a Zariski-dense subset of parameter space---we 
have shown that exact linear lumping yields only reductions that are, in a precise sense, 
forced by network structure, combining elimination of non-reactant species (Type 1) and projection 
along stoichiometric first integrals (Type 2). This result extends beyond mass 
action kinetics to product-form rate laws including Michaelis--Menten and Hill kinetics, 
suggesting that the limitation is fundamental to the algebraic structure of reaction 
networks rather than an artifact of the mass action assumption.

The generic case provides theoretical context for computational lumping methods such as 
CLUE~\cite{OVPT2021} and ERODE~\cite{CardelliTTV2017}. These algorithms seek structural 
(parameter-independent) lumpings, which correspond precisely to our generic setting. Our 
results explain both the scope and limitations of such approaches: when these algorithms 
find only modest reductions, it is because the generic theory permits nothing more. Moreover, our results may be used to greatly simplify the necessary computations in these algorithms.

On the other hand, the critical parameter framework addresses what structural methods cannot---identifying 
parameter-dependent reductions that become available when rate constants satisfy special 
algebraic relations. We have shown for mass action systems that the determination of critical parameter values and corresponding lumping matrices amounts to solving a finite system of polynomial equations for the rate parameters and matrix entries. From a practical perspective feasibility problems arise for larger systems, as could be expected, but there are escape routes to obtain at least partial solutions. 

Turning to proper lumpings -- a classical concept used by various communities -- we introduce a new aspect by exhibiting a relation to permutation symmetries and reduction by invariants. 

The relationship between lumping and time-scale methods such as the quasi-steady-state 
approximation (Segel and Slemrod\cite{SegelSlemrod1989}, Shoffner \cite{Shoffner}), computational singular 
perturbation (Lam and Goussis \cite{LamGoussis1994}), and intrinsic low-dimensional manifolds (Maas and Pope \cite{MaasPope1992}) 
deserves emphasis. These approaches are complementary rather than competing, even if they overlap in some instances (see Example \ref{ex:pqr}); thus, critical parameters for linear lumping may at the same time be critical parameters for singular perturbations (Tikhonov-Fenichel parameter values; see Goeke et al.~\cite{gwz}). For complex networks exhibiting both structural symmetries and time-scale 
separation, one may apply lumping first to remove algebraic redundancies, then 
time-scale methods to address the remaining slow-fast decomposition.

The critical parameter perspective has practical implications beyond exact reduction. When system parameters lie $\varepsilon$-close to a critical parameter value, the corresponding lumping map provides an approximate reduction with error of order $\varepsilon$ (at least) over finite time intervals. Numerical results indicate that for many systems such an error bound is also valid for all positive times, but there remains a need for a rigorous statement and proof. Moreover, by arguments analogous to those characterizing quasi-steady state (Goeke et al.~\cite{GWZ2015}), approximate lumping with arbitrarily small error implies proximity to critical parameters. This observation unifies exact and approximate lumping within a single geometric framework, where the critical variety in parameter space serves as an organizing structure.

Several directions merit further investigation. The present work concerns linear lumping
maps; the theory of nonlinear lumping, while developed by Li, Rabitz, and T\'oth~\cite{LRT1994} 
for fixed parameters, has not been explored from the generic/critical parameter 
perspective. Stochastic reaction networks modeled as continuous-time Markov chains 
involve related but distinct lumping conditions, and recent work on exact stochastic 
lumping~\cite{CardelliStochastic2021} suggests that parameter-dependence there may exhibit 
different structure. Finally, while our approach is algorithmic in principle, efficient 
implementations for large networks would benefit from tools in numerical algebraic 
geometry, particularly for computing the irreducible components and dimensions of 
critical varieties.

\section*{Acknowledgments}

We thank two anonymous reviewers for constructive comments and suggestions, in particular toward a clarification of the presentation.\\
During the preparation and revision of this manuscript, AI-assisted technologies (various versions 
utilized by ChatGPT and Claude between late 2025 and 2026) were used to verify algebraic 
computations, review of Python code and reproducibility documentation, and refine 
the writing in some parts of the manuscript. The authors assume full responsibility for
the accuracy and integrity of all content.

\section*{Code Availability}
A Python implementation of the algorithms and examples presented in this paper 
is available at \url{https://github.com/santiago-schnell/lumping-of-reaction-networks}. 
The package provides tools for defining reaction networks, identifying generic and 
critical scenarios, and validating reductions numerically.

\section{Appendix}
\subsection{Computations for subsection \ref{sec:replication}}\label{sec:61comps}
Evaluating the critical parameter condition $T\,DF(x,k)\,B=0$ (for all x) amounts to 
determining the zeros of the corresponding ideal in $\mathbb C[p,q,r, k_i]$.  Computing 
with the procedure {\tt minAssChar} of  {\sc Singular} \cite{sing} one obtains the 
minimal associate primes which define  22  irreducible components of the variety of 
the ideal. We include  the output below. Components [20], [21] and [22] correspond to 
the cases obtained by direct inspection in subsection \ref{sec:replication}.
\begin{verbatim}[1]:
   _[1]=km5
   _[2]=km3
   _[3]=km1
   _[4]=k4
   _[5]=k3
   _[6]=k2
   _[7]=k1
   _[8]=r*k5-p*km2+r*km2-q*km4+r*km4-2*k5
[2]:
   _[1]=km3
   _[2]=km1
   _[3]=k4
   _[4]=k3
   _[5]=k2
   _[6]=k1
   _[7]=r-2
   _[8]=p*km2+q*km4-2*km2-2*km4
[3]:
   _[1]=km5
   _[2]=km3
   _[3]=km1
   _[4]=k4
   _[5]=k3
   _[6]=k1
   _[7]=p-r
   _[8]=r*k5-q*km4+r*km4-2*k5
[4]:
   _[1]=km4
   _[2]=km3
   _[3]=km1
   _[4]=k4
   _[5]=k3
   _[6]=k1
   _[7]=r-2
   _[8]=p-2
[5]:
   _[1]=km5
   _[2]=km3
   _[3]=km1
   _[4]=k3
   _[5]=k2
   _[6]=k1
   _[7]=q-r
   _[8]=r*k5-p*km2+r*km2-2*k5
[6]:
   _[1]=km3
   _[2]=km2
   _[3]=km1
   _[4]=k3
   _[5]=k2
   _[6]=k1
   _[7]=r-2
   _[8]=q-2
[7]:
   _[1]=km5
   _[2]=km3
   _[3]=km1
   _[4]=k5
   _[5]=k3
   _[6]=k1
   _[7]=q-r
   _[8]=p-r
[8]:
   _[1]=km3
   _[2]=km1
   _[3]=k3
   _[4]=k1
   _[5]=r-2
   _[6]=q-2
   _[7]=p-2
[9]:
   _[1]=km3
   _[2]=km1
   _[3]=k5
   _[4]=k3
   _[5]=k1
   _[6]=r
   _[7]=q
   _[8]=p
[10]:
   _[1]=km5
   _[2]=km3
   _[3]=k4
   _[4]=k3
   _[5]=k2
   _[6]=p-1
   _[7]=r*k5+r*km2-q*km4+r*km4-2*k5-km2
[11]:
   _[1]=km3
   _[2]=k4
   _[3]=k3
   _[4]=k2
   _[5]=r-2
   _[6]=p-1
   _[7]=q*km4-km2-2*km4
[12]:
   _[1]=km5
   _[2]=km3
   _[3]=k4
   _[4]=k3
   _[5]=r-1
   _[6]=p-1
   _[7]=q*km4+k5-km4
[13]:
   _[1]=km5
   _[2]=km3
   _[3]=k3
   _[4]=k2
   _[5]=q-r
   _[6]=p-1
   _[7]=r*k5+r*km2-2*k5-km2
[14]:
   _[1]=km3
   _[2]=km2
   _[3]=k3
   _[4]=k2
   _[5]=r-2
   _[6]=q-2
   _[7]=p-1
[15]:
   _[1]=km5
   _[2]=km1
   _[3]=k4
   _[4]=k2
   _[5]=k1
   _[6]=q-1
   _[7]=r*k5-p*km2+r*km2+r*km4-2*k5-km4
[16]:
   _[1]=km1
   _[2]=k4
   _[3]=k2
   _[4]=k1
   _[5]=r-2
   _[6]=q-1
   _[7]=p*km2-2*km2-km4
[17]:
   _[1]=km5
   _[2]=km1
   _[3]=k4
   _[4]=k1
   _[5]=q-1
   _[6]=p-r
   _[7]=r*k5+r*km4-2*k5-km4
[18]:
   _[1]=km4
   _[2]=km1
   _[3]=k4
   _[4]=k1
   _[5]=r-2
   _[6]=q-1
   _[7]=p-2
[19]:
   _[1]=km5
   _[2]=km1
   _[3]=k2
   _[4]=k1
   _[5]=r-1
   _[6]=q-1
   _[7]=p*km2+k5-km2
[20]:
   _[1]=km5
   _[2]=k4
   _[3]=k2
   _[4]=q-1
   _[5]=p-1
   _[6]=r*k5+r*km2+r*km4-2*k5-km2-km4
[21]:
   _[1]=km2+km4
   _[2]=k4
   _[3]=k2
   _[4]=r-2
   _[5]=q-1
   _[6]=p-1
[22]:
   _[1]=km5
   _[2]=k5
   _[3]=r-1
   _[4]=q-1
   _[5]=p-1
\end{verbatim}
\subsection{Computations for subsection \ref{subsec:enzyme}}\label{sec:62comp}
A {\sc Singular} code for the ideal decomposition  using the routine \texttt{facstd} 
is as follows:
{\small

\medskip
      
ring R= 0,(t1,t2,t3,t4,t5,t6,t7,t8,t9,k1, k2, k3, k4, k5, km1,km2, km3,km4,km5),dp;

option(redSB);

ideal g=$(k1 + k3)*(1 + t1), -((k1 + k3)*t1*(1 + t1)), km5*(t1 - t2 + t3), km1*t4 + km1*t1*t4 + k2*t2*t4 + km3*t7 + km3*t1*t7 + k4*t2*t7, 
 -((k1 + k3)*(1 + t1)*t2), -(k5*t1) + k5*t2 + km2*t2 + km4*t2 - k5*t3 + km1*t5 + km1*t1*t5 + k2*t2*t5 + km3*t8 + km3*t1*t8 + k4*t2*t8, 
 km5*t1 - km5*t2 - k1*t3 - k3*t3 + km5*t3 - k1*t1*t3 - k3*t1*t3, km1*t6 + km1*t1*t6 + k2*t2*t6 + km3*t9 + km3*t1*t9 + k4*t2*t9, 
 -k1 + k1*t4 + k3*t4, -(t1*(-k1 + k1*t4 + k3*t4)), km5*(t4 - t5 + t6), -(k2*t4) - km1*t4 + km1*t4^2 + k2*t4*t5 + km3*t4*t7 + k4*t5*t7, 
 -(t2*(-k1 + k1*t4 + k3*t4)), -km2 - k5*t4 - k2*t5 + k5*t5 - km1*t5 + km2*t5 + km4*t5 + km1*t4*t5 + k2*t5^2 - k5*t6 + km3*t4*t8 + 
  k4*t5*t8, k1*t3 + km5*t4 - k1*t3*t4 - k3*t3*t4 - km5*t5 + km5*t6, -(k2*t6) - km1*t6 + km1*t4*t6 + k2*t5*t6 + km3*t4*t9 + k4*t5*t9, 
 -k3 + k1*t7 + k3*t7, -(t1*(-k3 + k1*t7 + k3*t7)), km5*(t7 - t8 + t9), -(k4*t7) - km3*t7 + km1*t4*t7 + km3*t7^2 + k2*t4*t8 + k4*t7*t8, 
 -(t2*(-k3 + k1*t7 + k3*t7)), -km4 - k5*t7 + km1*t5*t7 - k4*t8 + k5*t8 + km2*t8 - km3*t8 + km4*t8 + k2*t5*t8 + km3*t7*t8 + k4*t8^2 - 
  k5*t9, k3*t3 + km5*t7 - k1*t3*t7 - k3*t3*t7 - km5*t8 + km5*t9, km1*t6*t7 + k2*t6*t8 - k4*t9 - km3*t9 + km3*t7*t9 + k4*t8*t9;$
list j=facstd(g);
j;


The computation takes a few minutes.  Altogether there are  92 components in the output; 
the first five of them  are shown here:}
\begin{verbatim}
    [1]:
   _[1]=km5
   _[2]=km4
   _[3]=km3
   _[4]=km2
   _[5]=km1
   _[6]=k5
   _[7]=k4
   _[8]=k3
   _[9]=k2
   _[10]=k1
[2]:
   _[1]=km5
   _[2]=km4
   _[3]=km3
   _[4]=km2
   _[5]=k5
   _[6]=k4
   _[7]=k3
   _[8]=k1
   _[9]=t8*k2+t7*km1
   _[10]=t5*k2+t4*km1-k2-km1
   _[11]=t2*k2+t1*km1+km1
   _[12]=t5*t7-t4*t8-t7+t8
   _[13]=t2*t7-t1*t8-t8
   _[14]=t2*t4-t1*t5+t1-t2-t5+1
[3]:
   _[1]=km5
   _[2]=km4
   _[3]=km3
   _[4]=km1
   _[5]=k5
   _[6]=k4
   _[7]=k3
   _[8]=k1
   _[9]=t8
   _[10]=t5-1
   _[11]=t2
[4]:
   _[1]=km5
   _[2]=km3
   _[3]=km1
   _[4]=k4
   _[5]=k3
   _[6]=k2
   _[7]=k1
   _[8]=t7*k5-t8*k5+t9*k5-t8*km2-t8*km4+km4
   _[9]=t4*k5-t5*k5+t6*k5-t5*km2-t5*km4+km2
   _[10]=t1*k5-t2*k5+t3*k5-t2*km2-t2*km4
   _[11]=t2*t4-t1*t5-t3*t5+t2*t6+t2*t7-t1*t8-t3*t8+t2*t9+t1-t2+t3
   _[12]=t5*t7*km2-t4*t8*km2-t6*t8*km2+t5*t9*km2+t5*t7*km4-t4*t8*km4
   -t6*t8*km4+t5*t9*km4-t7*km2+t8*km2-t9*km2+t4*km4-t5*km4+t6*km4
   _[13]=t2*t7*km2-t1*t8*km2-t3*t8*km2+t2*t9*km2+t2*t7*km4-t1*t8*km4
   -t3*t8*km4+t2*t9*km4+t1*km4-t2*km4+t3*km4
[5]:
   _[1]=km5
   _[2]=km3
   _[3]=k4
   _[4]=k3
   _[5]=k1
   _[6]=t7*k5-t8*k5+t9*k5-t8*km2-t8*km4+km4
   _[7]=t4*k5-t5*k5+t6*k5-t5*km2-t5*km4+km2
   _[8]=t1*k5-t2*k5+t3*k5-t2*km2-t2*km4
   _[9]=t8*k2+t7*km1
   _[10]=t5*k2+t4*km1-k2-km1
   _[11]=t2*k2+t1*km1+km1
   _[12]=t5*t7-t4*t8-t7+t8
   _[13]=t2*t7-t1*t8-t8
   _[14]=t3*t5-t2*t6+t3*t8-t2*t9-t3-t5-t8+1
   _[15]=t3*t4-t1*t6+t3*t7-t1*t9-t3-t4-t6-t7-t9+1
   _[16]=t2*t4-t1*t5+t1-t2-t5+1
   _[17]=t6*t8*km2-t5*t9*km2+t6*t8*km4-t5*t9*km4+t9*km2-t4*km4+t5*km4-t6*km4-t7*km4+t8*km4
   _[18]=t3*t8*km2-t2*t9*km2+t3*t8*km4-t2*t9*km4-t8*km2-t1*km4+t2*km4-t3*km4-t8*km4
   _[19]=t6*t8*k5-t5*t9*k5+t9*k5-t5*km4-t8*km4+km4
   _[20]=t3*t8*k5-t2*t9*k5-t8*k5-t2*km4
   _[21]=t6*t7*km1*km2-t4*t9*km1*km2+t6*t7*km1*km4-t4*t9*km1*km4+t9*km1*km2+t4*k2*km4
   +t6*k2*km4+t7*k2*km4+t9*k2*km4+t4*km1*km4+t7*km1*km4+t9*km1*km4-k2*km4-km1*km4
   _[22]=t3*t7*km1*km2-t1*t9*km1*km2+t3*t7*km1*km4-t1*t9*km1*km4-t7*km1*km2-t9*km1*km2
   +t1*k2*km4+t3*k2*km4+t1*km1*km4-t7*km1*km4-t9*km1*km4+km1*km4
\end{verbatim}


\end{document}